\documentclass[11pt,a4paper]{amsart}
\input epsf

\usepackage{graphicx, amsmath, amssymb, amscd, amsthm, euscript, amsfonts, xcolor, latexsym, amsbsy}

\newcommand{\C}{\mathbb C}

\renewcommand{\P}{\mathbb{P}}

\newcommand{\vo}{\mathbf{0}}

\newcommand{\nn}{\mathbf{n}}
\newcommand{\mm}{\mathbf{m}}
\newcommand{\pp}{\mathbf{p}}
\newcommand{\qq}{\mathbf{q}}

\newcommand{\hyp}{\mathbb{H}^2_{\mathbb{C}}}

\newcommand{\bpm}{\begin{pmatrix}}
\newcommand{\epm}{\end{pmatrix}}

\newcommand{\sh}{{\sinh\,}}

\newtheorem{thm}{Theorem}[section]
\newtheorem{cor}[thm]{Corollary}
\newtheorem{prop}[thm]{Proposition}
\newtheorem{lem}[thm]{Lemma}
\theoremstyle{definition}
\newtheorem{defn}[thm]{Definition}

\newtheorem{rem}[thm]{Remark}

\title{Tubes in  Complex Hyperbolic Manifolds}

\author[Ara Basmajian]{Ara Basmajian}
\thanks{AB  supported by PSC CUNY Award 65245-00 53
and partially supported by  Simons  Collaboration Grant (359956, A.B.)}
\address[Ara Basmajian]{The Graduate Center, CUNY, 365 Fifth Ave., N.Y., N.Y., 10016, USA,
and Hunter College, CUNY, 695 Park Ave., N.Y., N.Y., 10065, USA}
\email{abasmajian@gc.cuny.edu}

\author[Youngju Kim]{Youngju Kim}
\thanks{YK  was supported by a  National Research Foundation of Korea(NRF) grant funded by the Korea government(MSIT) (No.  NRF-2021R1F1A1045633).
}
\address[Youngju Kim]{Konkuk University, Department of Mathematics Education, Seoul, 05029, Republic of Korea}
\email{geometer2@konkuk.ac.kr}

\keywords{collar lemma,  complex geodesic surface, complex hyperbolic manifold, tubular neighborhood theorem}
\subjclass[2020]{Primary 53C55, 22E40; Secondary 30F40}

\begin{document}
\bibliographystyle{h-elsevier2}

\begin{abstract}
We prove a tubular neighborhood theorem for an embedded complex  geodesic surface in a complex hyperbolic 2-manifold where
the width of the tube depends only on   the Euler characteristic of the embedded surface.  We give an  explicit estimate for this width. We supply two applications of the tubular neighborhood theorem, the first is a  lower  volume bound for  such manifolds.   The second is an upper bound on the
first eigenvalue of the Laplacian in terms of the geometry of the manifold.
Finally, we prove a geometric combination theorem for two Fuchsian subgroups of ${\rm PU}(2,1)$. Using this combination theorem, we show that the optimal width size of a tube about an embedded  complex geodesic surface is asymptotically bounded between $\frac{1}{|\chi |}$ and
$\frac{1}{\sqrt{|\chi |}}$.
\end{abstract}

\maketitle

\section{Introduction and statement of results}\label{sec: intro}

The celebrated collar lemma guarantees that a simple closed geodesic
on a hyperbolic surface has a collar (tubular neighborhood) whose width  only depends on its length \cite{Buser, Buser-book, Chavel-Feldman, Keen}. In fact relaxing the curvature condition to pinched negative
also yields similar theorems in dimension two \cite{Buser, Chavel-Feldman}.  The main point  is  that the width of the collar  only depends on  the local geometry of the geodesic. Namely,  width is a function of the  geodesic  length   and not  the ambient geometry of the surface. Another   generalization of the collar  lemma, called the tubular neighborhood theorem \cite{B-tubular},  is to embedded totally geodesic hypersurfaces (real codimension one) in hyperbolic $n$-manifolds.
In such a setting length is replaced by
the  $(n-1)$-dimensional volume of the hypersurface and the tubular width function depends only on this $(n-1)$-volume.
Such a universal statement fails to hold if we leave (real) codimension one.
For example,  while the Margulis lemma insures  that  a sufficiently short geodesic has a  tubular neighborhood   whose radius is a function of the length \cite{Brooks-M},
there is no version of the collar lemma that holds  for long  simple closed geodesics in hyperbolic 3-manifolds (see section  \ref{sec: examples} for examples of such manifolds).

 A complex hyperbolic $2$-manifold is the  quotient of complex hyperbolic
$2$-space, $\hyp$, by a torsion-free discrete subgroup $\Gamma$ of ${\rm PU}(2,1)$. If moreover, $\Gamma$ stabilizes a complex line $L$ in
$\hyp$, then we say that $(\Gamma, L)$  is $\mathbb{C}$-Fuchsian, and the quotient
$L/\Gamma$ is a {\it complex geodesic surface}. The complex geodesic surface is totally geodesic of constant curvature -1.
  In this paper,  we prove  a   tubular neighborhood theorem for  embedded  complex geodesic surfaces in a complex hyperbolic $2$-manifold. While such an embedded  surface has real codimension two,
the  ambient complex geometry exhibits enough of the features of  real geodesics in a real hyperbolic surface to allow us to construct a tubular neighborhood.

Throughout this work we use the notation $S \hookrightarrow X$
to denote an {\it embedded complex geodesic surface  $S$ in  a complex hyperbolic 2-manifold $X$}. We also assume without further mention   that any two connected lifts of $S$ to the universal cover $\hyp$ do not  meet at infinity. Equivalently $S$  is not asymptotic to itself; for otherwise, there is no  tubular neighborhood of any width.

\medskip\noindent\textbf{Theorem A.}[Theorems \ref{thm:tubular} and \ref{thm:2-embedded}]
\textit{There exists a positive function $c(x)$ so that
  any finite area   surface $S \hookrightarrow X$  has a tubular neighborhood  of width at least $c(|\chi |)$  where
  $\chi$  is the Euler characteristic of $S$. The function  can be taken to be
\begin{equation}\label{eqn:width-bound-thm}
c(|\chi |) =\frac{1}{4}\log\left( \frac{1}{\pi |\chi |}+1 \right).
\end{equation}
Moreover,  two disjoint complex  geodesic  surfaces  $S_1$ and $S_2$ have disjoint tubular neighborhoods of width
$
\frac{1}{8}
 \log \left(\frac{1}{\pi{\rm max}\{|\chi_1  |,|\chi_2 | \}}+1\right),
$
where $\chi_k$ is the Euler characteristic of $S_{k}$, $k=1,2$.}
\vskip10pt

\noindent\textbf{Remark.}
We note that there is a
shortest geodesic from $S$ to itself which can not be freely homotoped into $S$. Such a geodesic meets $S$ perpendicularly at its start and end points and is usually referred to as an {\it orthogeodesic}. The content of Theorem A is that this shortest geodesic has a lower length bound that only depends on the Euler characteristic of $S$ and not on the manifold $X$.  If  $S$ is a closed surface  the   assumption  that it is not asymptotic to itself is automatic.

We use the notation  ${\rho_{X}(\cdot,\cdot)}$ to denote complex hyperbolic distance on $X$.
\vskip10pt

\noindent\textbf{Corollary B.}[Corollary \ref{cor:main}]
\textit{
Suppose that $X$ is a complex hyperbolic $2$-manifold and
 $S_1, S_2, \dots, S_n$ are pairwise disjoint embedded complex geodesic surfaces in $X$ satisfying ${\rho_{X}}(S_k, S_j)>0$, for $k \neq j$.
Then \begin{equation}
{\rm vol}(X) \geq  {4\pi^{2}}n | \chi |\Biggl[ \cosh^4 \bigg( {\frac{1}{16}\log\left(\frac{1}{\pi |\chi |} +1 \right)}\bigg) - 1 \Biggr]
\end{equation}
where $| \chi | = {\rm max}\left\{ |\chi (S_{1}) |,  \dots ,| \chi (S_n) | \right\}$.
}
\vskip10pt

Denote  the tubular neighborhood of an embedded surface $S$
by $\mathcal{N}_{c(|\chi|)}$. As an application of the tubular neighborhood theorem we obtain a bound on the first eigenvalue $\lambda (X)$ of the Laplacian in terms of the volume of the manifold $X$, the tubular neighborhood width $c(|\chi |)$ of the surface
$S$,  and the volume of the tubular neighborhood
$\mathcal{N}_{c(|\chi|)}$.
\vskip10pt

\noindent\textbf{Theorem C.}[Theorem \ref{thm: eigenvalue upper bound}]
\textit{
Suppose we have an embedded  compact complex geodesic  surface
$S \hookrightarrow X$, where $X$ is a closed  complex hyperbolic 2-manifold. Then
\begin{equation}
\lambda (X) \leq \frac{{\rm vol}(\mathcal{N}_{c(|\chi |)})}{(c(|\chi |))^{2}[{\rm vol}(X)-{\rm vol}(\mathcal{N}_{c(|\chi | )})]}.
\end{equation}
}
\vskip10pt

Using Theorem C, we are able to give an explicit upper bound  for
$\lambda (X)$ in terms of the volume of $X$ and the Euler characteristic of $S$. See Corollary  \ref{cor: explicit estimate}.

We next give a complex hyperbolic version of the combination theorem.
It is the analogue  of   Theorem 1.4 in  \cite{B-tubular}.
Setting  $s(x) = 2{\rm arcsinh}\left( \frac{1}{{\rm sinh}\frac{x}{2}} \right)$, we have

\vskip10pt
\noindent\textbf{Theorem D.}[Theorem \ref{thm: combination theorem}]
\textit{Let $(\Gamma_{1}, L_{1})$ and $(\Gamma_{2}, L_{2})$ be
$\mathbb{C}$-Fuchsian groups where $L_1$ and $L_2$ are disjoint, and
let $p_1 \in L_1$ and $p_2 \in L_2$ be the endpoints of the unique orthogeodesic from $L_1$ to $L_2$.
Suppose
\begin{equation}
s({\rm inj}(p_1))+s({\rm inj}(p_2)) < \rho (L_1,L_2).
\end{equation}
Then  $\Gamma =<\Gamma_1,\Gamma_2>$  is a discrete subgroup of
${\rm PU}(2,1)$ where
\begin{enumerate}
\item  $\Gamma$  is abstractly  the free
product $\Gamma_1 \ast \Gamma_2$
\item $L_{1}/\Gamma_{1}$ and $L_{2}/\Gamma_{2}$ are embedded complex geodesic surfaces in $X=\mathbb{H}_{\mathbb{C}}^2 /\Gamma$
\item  $\rho_{X} (S_{1},S_{2})=\rho (L_1,L_2).$
\end{enumerate}}
\vskip10pt
In the above statement we have used $inj (p_i)$ to denote the {\it injectivity radius}  of the projection of  $p_{i}$ in the surface  $L_{i}/\Gamma_{i}$, for $i=1,2.$

\vskip10pt

Finally, we use Theorem D to give bounds on the optimal tube function,
\begin{equation*}
d_{\chi}=\sup \{\epsilon : \text{if } S \hookrightarrow X
\text{~and~ $\chi(S)=\chi$, then $S$ has a tube of width } \epsilon   \}.
\end{equation*}

\vskip10pt

\noindent\textbf{Corollary E.}[Corollary \ref{cor: optimal tube function}]
\textit{The optimal tube  function $d_{\chi}$ satisfies
\begin{equation}
\frac{1}{4}\log\left( \frac{1}{\pi |\chi |}+1 \right)\leq d_{\chi} \leq
s\left({\rm arccosh}\left(\cot \frac{\pi/2}{|\chi | +2}\right)\right).
\end{equation}
In particular,
\begin{equation}
\frac{1}{4 \pi |\chi |} \lesssim d_{\chi} \lesssim
\frac{2\sqrt{\pi}}{\sqrt{|\chi |}}, \text{ as } |\chi| \rightarrow \infty.
\end{equation}
}
Here  the notation $f(x) \lesssim g(x)$ means there exists a constant $C >0$  so that $f(x) \leq C g(x)$, for $x$ large.

\subsection{A little history} Besides the references mentioned earlier in this section for codimension one, in \cite{B-Br-Ha-Me} it is shown that in a pinched negatively  curved  manifold,  an embedded  totally geodesic
hypersurface has a tubular neighborhood whose width is explicitly given as a function of the pinching constants, dimension, and the $(n-1)$-volume of the hypersurface.  Other  collar lemma type theorems  and related works  include
 \cite{B-pants, B-Bulletin, B-stable, Bas-Miner, Bey-Poz, Cao-Parker,
 Dryden-Parlier, Gallo, Gilman, Go-Ka-Le,  Kim, Kojima, Lee-Zhang, Markham-Parker, Parker-volume, Parlier, Randol}.
 We emphasize that the collar lemmas about geodesics  in higher dimensions   only hold for short geodesics.

\subsection{Outline of proof}

First we give an outline of the ideas involved in showing the existence of a  tubular neighborhood function.
Suppose $S$ is an embedded  complex geodesic surface in the  complex hyperbolic $2$-manifold $X=\mathbb{H}_{\mathbb{C}}^2 /\Gamma$.  Let $\delta$ be  the shortest orthogeodesic from $S$ to itself, and let $d >0$ be its length.  Hence $S$ has a tubular neighborhood of width
$\frac{d}{2}$. We proceed to find a lower bound on $d$ which only depends on the Euler characteristic of $S$.

The first step is to consider a  lift of $\delta$ to
$\mathbb{H}_{\mathbb{C}}^{2}$ (we continue to call it $\delta$).
This lift $\delta$ is the common orthogonal  between  two lifts  of $S$,
 say $L_1$ to $L_2$,  where $L_1$ and $L_2$ are complex lines.  Next we orthogonally project $L_2$ to  the disc $D$ in $L_1$. A  crucial fact is that the size of this projection only depends on  $d$,  the length of
 $\delta$.  Now  the $\mathbb{C}$-Fuchsian subgroup
 ${\rm Stab}_\Gamma(L_{1})$ moves  $D$ around in $L_1$.  Due to the presence of holonomy in ${\rm Stab}_\Gamma(L_{1})$  it is possible that there are translates of $D$ that intersect $D$.
   Next  we construct an embedded wedge
 $W \subset X$ anchored on   $D$ of width $\frac{d}{2}$ (see Figure   \ref{fig:wedge}). The wedge $W$ is determined by three parameters:
 the size (radius)  of its anchor, its width $\frac{d}{2}$, and its sector angle.
 We next prove a technical lemma we call the holonomy  angle lemma (Lemma \ref{lem:holonomy}) which  guarantees that the sector angle is lower bounded by a function of $d$. Hence the wedge $W$
 is a function of $d$ allowing  us to show that it is embedded in $X$.
 In fact, $W  \subset N_{\frac{d}{2}}(S)$, where
$N_{\frac{d}{2}}(S)$ denotes the  $\frac{d}{2}$-neighborhood of $S$.
 Finally comparing  the volume of $W$
with the volume of $N_{\frac{d}{2}}(S)$ and noting that    as $d$ gets smaller the volume inequality, ${\rm vol}(W)\leq {\rm vol}(N_{\frac{d}{2}}(S))$,  does not hold guarantees that there is a critical lower bound on $d$. That will be our function $c(|\chi |)$,  where $|\chi |$ is the absolute value of the Euler characteristic of $S$.
\vskip10pt

\subsection{Section plan and notation}
Theorem A  is a  consolidation of Theorems \ref{thm:tubular} and  \ref{thm:2-embedded}, and Corollary B is Corollary \ref{cor:main} where
we have used the Euler characteristic in place of area for a  totally geodesic complex surface.
Sections \ref{sec: cx hyp geom prelim}-\ref{sec:vol-wedge}
 cover general facts about complex lines and volumes.
Sections \ref{sec: tubes} and \ref{sec: Two embedded surfaces} have the proofs of the tubular neighborhood theorem and its corollary.
In section \ref{sec: eigenvalue},  we give as a consequence of the tubular neighborhood theorem a   bound on the first non-zero eigenvalue of the Laplacian in terms of the geometry.
In section \ref{sec: examples} we supply examples to show that the tubular neighborhood theorem does not hold for geodesics in a hyperbolic
$3$-manifold.  Also, in  section \ref{sec: examples} we supply references  to complex hyperbolic  manifolds  containing  complex  geodesic surfaces.   In section \ref{combination thm and tube bounds} we prove a geometric version of a combination theorem and use it  to give bounds on the optimal tube function. Throughout this paper we use the words, collar, tube, and tubular neighborhood interchangeably.

Table \ref{Table:notation}
is a guide to the location of  the important  formulas and commonly used notations  in the paper.  For basic  references we refer to \cite{Buser-book, Goldman, Parker}.

\begin{table}
\begin{tabular}{| l | c | c |}
\hline
{\bf Definition} & {\bf Notation} &{\bf Section} \\
 \hline
 area of a disc in a complex line& ${\rm Area}(D_r)$ & Equation (\ref{eqn:area-disc}) \\
\hline
collar function & $c(A)$ & Equation (\ref{eqn:width-bound})  \\
\hline
complex hyperbolic space & $\hyp$ & Equation(\ref{def:hyperbolic-space}) \\
 \hline
 complex line stabilizer in ${\rm PU}(2,1)$& ${\rm Stab}(L)$ & Equation (\ref{def:stab-1})  \\
\hline
complex line stabilizer in   $\Gamma$ & ${\rm Stab}_\Gamma(L)$ & Equation (\ref{def:stab-2}) \\
\hline
 embedded complex geodesic surface $S$ in X &$S \hookrightarrow X$
 &Section \ref{sec: intro}    \\
 \hline
  holonomy & hol & Subsection \ref{sec:holonomy} \\
\hline
hyperbolic distance &  $\rho(\pp, \qq)$ & Equation (\ref{def:hyp-metric}) \\
\hline
injectivity radius &${\rm inj}(\cdot)$
& Subsection \ref{subsec: comb thm}\\
\hline
 $\cosh^2{ \left( \frac{\rho(L_1, L_2)}{2} \right)} $ & $N(L_{1},L_{2})$ & Proposition \ref{thm:polar-formula} \\
\hline
 optimal tube  function &$d_{\chi}$  & Section \ref{sec: intro}  \\
\hline
 orthogonal projection & $\Pi_L(M)$  & Definition \ref{def:ortho-proj} \\
\hline
polar vector with  respect to a complex line & &Definition \ref{def:polar} \\
\hline
 radius of projection & $s(d)$ &Equation (\ref{formula:radius-projection}) \\
\hline
 tubular neighborhood& $\mathcal{N}(S)$ & Theorem \ref{thm:tubular}  \\
\hline
volume form & {\rm dvol} & Section \ref{sec:vol-wedge} \\
\hline
 wedge & $W(s, \epsilon, \psi)$ & Definition \ref{def:wedge} \\
\hline
\end{tabular}
\vskip10pt
\caption{Notation}
\label{Table:notation}
\end{table}

\subsection*{Acknowledgement.}
 This work began   during several
visits to the   Korea Institute for Advanced Study in Seoul, South Korea.
The authors would like to thank the institute for their support.
The authors would also like to thank Elisha Falbel, Julien Paupert, and Mahmoud Zeinalian  for  helpful conversations.


\section{Complex hyperbolic geometry preliminaries}
\label{sec: cx hyp geom prelim}

In this section we set notation and discuss  some of the basics on complex hyperbolic geometry which we will be using later.
 As a basic reference we refer to  the  book \cite{Goldman},  and the notes \cite{Parker}.   In particular, the Riemannian metric on complex hyperbolic $2$-space is normalized as in the  references  \cite{Goldman, Parker}  to have sectional curvature,  $-1 \leq K \leq -1/4$. The extreme curvatures are realized by totally geodesic planes. Complex lines (to be defined below) having curvature $-1$ and totally real planes having   curvature
 $-1/4$.

Throughout this paper we will use the ball model of the complex hyperbolic 2-space.
That is, we consider complex $3$-space $\C^{2,1}$  with the first Hermitian form
$$J = \bpm 1&0&0 \\ 0&1&0 \\ 0&0&-1\epm$$
and hence
\begin{equation}\label{def:H-inner-product}
 \left< \mathbf{p}, \mathbf{q} \right> = \mathbf{q}^*J\mathbf{p}=p_1\overline{q_1}+p_2\overline{q_2} - p_3\overline{q_3}.
\end{equation}
We usually denote  vectors  in $\C^{2,1}$ with bold face letters.
%
%
%

Let $\mathbb{P}: \mathbb{C}^{2,1} \setminus \{0\} \rightarrow \mathbb{CP}^2$ be the canonical projection onto complex projective space.
On the chart of $\mathbb{C}^{2,1}$ with $p_3 \neq 0$, the projection map $\P$ is given by
\begin{equation}
(p_1, p_2, p_3) \mapsto \left( \frac{p_1}{p_3},  \frac{p_2}{p_3}\right).
\end{equation}
For any $p=(p_1, p_2) \in \mathbb{C}^2$,
 we lift the point $p$  to $\mathbf{p} = ( p_1, p_2, 1) \in \mathbb{C}^{2,1}$, called the \textit{standard lift of} p.
Then $\left< \mathbf{p},\mathbf{p} \right> = |p_1|^2 + |p_2|^2 -1$.
Thus, the ball model of complex hyperbolic $2$-space is
\begin{equation}\label{def:hyperbolic-space}
\mathbb{H}^2_{\mathbb{C}}  =  \{ (p_1, p_2)\in\mathbb{C}^2 \mid |p_1|^2 + |p_2|^2 < 1 \}
\end{equation}
and its boundary at infinity is
\begin{equation}
\partial\mathbb{H}^2_{\mathbb{C}} = \{ (p_1, p_2)\in\mathbb{C}^2 \mid |p_1|^2 + |p_2|^2 = 1 \}.
\end{equation}

The \textit{Bergman metric} $\rho$ on $\mathbb{H}^2_{\mathbb{C}}$ is defined as
\begin{equation}\label{def:hyp-metric}
 \cosh^2\left( \frac{{\rho}(p,q)}{2} \right)
  = \frac{\left< \mathbf{p}, \mathbf{q} \right> \left< \mathbf{q}, \mathbf{p} \right>}{\left< \mathbf{p}, \mathbf{p} \right>\left< \mathbf{q}, \mathbf{q} \right>}
\end{equation}
where $\mathbf{p}$ and $\mathbf{q}$ are the lifts of $p$ and $q \in \hyp$ respectively.
Let ${\rm SU}(2,1)$ be the group of unitary matrices which preserve the given Hermitian form with determinant $1$.
Then the group of holomorphic isometries of $\mathbb{H}^2_{\mathbb{C}}$ is
 ${\rm PU}(2,1) = {\rm SU}(2,1) / \{ I, {\omega}I, {\omega}^2I\}$,
 where $\omega = (-1+i\sqrt{3})/2$ is a cube root of unity.

\subsection{Complex lines and orthogonal projection}

 \begin{defn}\label{def:polar}
 Suppose $L \subseteq \C^{2,1}$  is a complex  two-dimensional subspace of $\C^{2,1}$.  A non-zero vector $\nn \in \C^{2,1}$ is said to be a {\it polar vector} for  $L$
if it is orthogonal to $L$ with respect to the Hermitian form  $J$.
\end{defn}

Given  a basis $\{\pp,\qq\}$ of $L$ where $\pp = (p_1,p_2, p_3), \qq=(q_1, q_2,q_3) \in \C^{2,1}$, a polar vector for $L$ can be taken to be the cross product $\nn$ of  $\pp$ and $\qq$,
\begin{equation}\label{polar}
 \nn=\left(~ \overline{p_2q_3}-\overline{p_3q_2},~ \overline{p_3q_1}-\overline{p_1q_3},~ \overline{p_2q_1}- \overline{p_1q_2} ~\right).
\end{equation}
A complex 2-dimensional subspace  $L$ of $\C^{2,1}$ projects to a {\it complex line}  of $\hyp$ if and only if the polar vector $\nn$ is a positive vector.
In the ball model, a complex line is precisely the (non-trivial) intersection of a complex one-dimensional affine  subspace of  $\C^2$ with $\hyp$.

\begin{prop}[\cite{Goldman}]\label{thm:polar-formula}
Let $L$ and $M$ be complex lines with polar vectors $\nn, \mm$ respectively, and set
\begin{equation*}
N(L, M) = \frac{\left< \nn, \mm\right>\left< \mm, \nn\right>}{\left< \nn, \nn\right>\left< \mm, \mm\right> }.
\end{equation*}
\begin{enumerate}
\item  $N(L, M)>1$, in which case $L$ and $M$ are ultraparallel and
\begin{equation*} \cosh^2{ \left( \frac{\rho(L, M)}{2} \right)} = N(L, M). \end{equation*}
\item  $N(L, M) =1 $, in which case $L$ and $M$ are asymptotic or coincide.
\item $N(L, M)<1$, in which case $L$ and $M$ intersect and
\begin{equation*} \cos^2\theta = N(L, M) \end{equation*}
 where $\theta$ is the angle of intersection between $L$ and $M$.
\end{enumerate}
\end{prop}

The subspace of complex lines inherits a topology from
$\mathbb{C}^{2,1}$, namely  a sequence of complex lines converges to a complex line if and only if the complex lines have unit polar vectors that converge in the topology of
$\mathbb{C}^{2,1}$. Though complex lines have real codimension greater  than one, the following proposition  says   that transversality of complex lines is a stable property.

\begin{prop}\label{prop:transversality}
Suppose two complex lines $L$ and $M$ intersect transversally.
Then any complex line near  $M$ also intersects  $L$.
\end{prop}

\begin{proof}
Without loss of generality, we may assume that
$L = \{ (0, w) \in \hyp \}$ with a unit polar vector $\nn = (1, 0, 0)$, and $\mm = ( m_1, m_2, m_3)$ the unit polar vectors of $M$.
Since $L$ and $M$ intersect each other transversally, $N(L, M) = |m_1|^2 <1$.

Let $0<\epsilon<1 - |m_1|$ be given.
If  $\qq = ( q_1, q_2, q_3)$ is a unit polar vector satisfying
\begin{equation}\label{eq: polar vectors near}
\sqrt{ |q_1 - m_1|^2 + |q_2 -m_2|^2 + |q_3 - m_3|^2 } <
\epsilon,
\end{equation}
then
\begin{equation}\label{eq:bounds}
|q_1| \leq |q_1 -m_1| + |m_1| <\epsilon + |m_1| < 1.
\end{equation}
Thus the complex line $Q$ whose polar vector is $\qq$ intersects $L$ because
\begin{equation*}
N(Q,L) = \left|\left< (q_1, q_2, q_3), (1, 0, 0) \right> \right|^2 = |q_1|^2 < 1.
\end{equation*}
\end{proof}


\begin{defn}\label{def:ortho-proj}
Given a complex line $L$ with a polar vector $\nn$, we define the \textit{orthogonal projection} $\Pi_L(p) : \hyp \rightarrow L$ to be
\begin{equation}
 \Pi_L(p) = \mathbb{P}\left( \pp - \frac{\left< \pp,\nn \right>}{\left< \nn,\nn\right> }\nn \right)
\end{equation}
for $p \in \hyp$.
\end{defn}

This projection is the nearest point projection map \cite{Goldman}.
For a complex line $L$ and  a subgroup $\Gamma < {\rm PU}(2,1)$,
        we denote the stabilizer of $L$ in ${\rm PU}(2,1)$  and in $\Gamma$ respectively, as
\begin{align}
  &{\rm Stab}(L) = \{ \gamma \in {\rm PU}(2,1) \mid \gamma(L) = L \}, \label{def:stab-1}\\
  &{\rm Stab}_\Gamma(L) = \{ \gamma \in \Gamma \mid \gamma(L) = L \}. \label{def:stab-2}
\end{align}

Here, we recall some geometric properties of the projection which we use later. For completeness  we supply   proofs.

\begin{prop}\label{prop:commute}
Let $L$ be a complex line and $\gamma \in {\rm Stab}(L)$. Then
\begin{equation}
 \Pi_L{\circ}\gamma(p) = \gamma{\circ}\Pi_L(p)
\end{equation}
for $p \in \hyp$.
\end{prop}

\begin{proof}
This follows from the fact that the geodesic from $p$ to $L$ which intersects  $L$ orthogonally is moved
by $\gamma$ to the geodesic from $\gamma (p)$ to $L$ which intersects $L$ orthogonally.
\end{proof}


\begin{prop}\label{prop:s}
If two disjoint complex lines $L_1$ and $L_2$ are a distance $d>0$ a part, then
the projection $\Pi_{L_1}(L_2)$ of $L_2$ into $L_1$ is a disc of radius $s(d) = 2\log{\coth{\frac{d}{4}}}$. Moreover, the function $s(d)$ is an involution.
\end{prop}

\begin{proof}
Without loss of generality, we may assume that $L_1 = \{ (0,w) \in \C^2 ~\mid~ |w|<1 \}$,
$M=\{ (z,0) \in \C^2 ~\mid~ |z| <1 \}$,
$L_1 \cap M = \{ \vo=(0,0) \}$, and the distance $d$ from $L_1$ to $L_2$ is realized by the geodesic segment connecting $\vo$ and $(x,0)$ in $M$ for $x = {\rm tanh}\frac{d}{2}$.
i.e, $M \cap L_2 = \{(x,0)\}$ (see Figure \ref{fig:setting}).
Let $f_x \in {\rm Stab}(M)$ be a loxodromic isometry mapping $\vo$ to $(x, 0)$,
\begin{equation}\label{def:nor-hyp-M}
f_x =\bpm \frac{1}{\sqrt{1-x^2}} & 0 & \frac{x}{\sqrt{1-x^2}} \\
     0 & 1 & 0 \\
      \frac{x}{\sqrt{1-x^2}} & 0 &   \frac{1}{\sqrt{1-x^2}} \epm.
\end{equation}
Applying Proposition \ref{prop:commute}, $$L_2 = \Pi_M^{-1}(x,0) = f_x{\circ}\Pi_M^{-1}(\vo) = f_x(L_1).$$
For  $(0, w) \in L_1$,
\begin{equation}
f_x(0,w,1) = \bpm \frac{1}{\sqrt{1-x^2}} & 0 & \frac{x}{\sqrt{1-x^2}} \\
                 0 & 1 & 0 \\
                \frac{x}{\sqrt{1-x^2}} & 0 &   \frac{1}{\sqrt{1-x^2}} \epm
            \bpm 0\\w\\1\epm
         = \bpm x\\w\sqrt{1-x^2}\\1\epm.
\end{equation}
Thus  we have,
$$L_2=\Bigg\{ (x, w\sqrt{1-x^2}) \in \hyp \mid |w|<1 \Bigg\}.$$

Therefore, the projection $\Pi_{L_1}(L_2)$ is given in coordinates,
\begin{equation}
 {\Pi}_{L_1}(x, w\sqrt{1-x^2})  =  (0, w\sqrt{1-x^2}).
\end{equation}
and it is a disc centered at $\vo$ with radius $s(d)$,
\begin{equation}\label{formula:radius-projection}
\begin{split}
s(d) & = \log\left( \frac{1+ \sqrt{1-x^2}}{1-\sqrt{1-x^2}} \right)
       = 2\log{\coth{\frac{d}{4}}}  
    = 2{\rm arcsinh}\left( \frac{1}{{\rm sinh}\frac{d}{2}} \right).\\
\end{split}
\end{equation}
\end{proof}

\begin{prop}
If two complex lines $L$ and $M$ intersect at a point $p$ with angle $\theta$,
then the projection $\Pi_{M}(L)$ of $L$ to $M$  is a disc  centered at $p$ with radius $\log\frac{1+|\cos\theta|}{1-|\cos\theta|}$.
\end{prop}

\begin{proof}
Without loss of generality, let $p =(0,0)$ and $M = \{ (z, 0) \mid |z| < 1 \}$ with a polar vector $\mm = (0,1,0)$.
If a complex line $L$ passes through $p$, then its polar vector is of the form $(a,b,0)\in \C^{2,1}$ for $b \neq 0$ and thus we  may normalize its polar vector  to be $\nn=(a, 1,0)$ for $a \in \C$.
Since any $q \in L$ 
satisfies  $\langle \qq, \nn \rangle = 0$ and $\langle \qq, \qq \rangle <0$,
$$L = \left\{ (z, -z\overline{a}) \in \hyp \mid z = re^{i\psi}, r < \sqrt{\frac{1}{1+|a|^2}}, 0\leq \psi \leq 2\pi \right\}.$$
For $ q = (z, -z\overline{a}) \in L$,  $\langle \qq, \mm \rangle = -z\overline{a}$ and
\begin{align*}
 \Pi_{M}(\qq)= \mathbb{P}\left( \qq - \frac{\left< \qq, \mm \right>}{\left< \mm, \mm \right> }\mm \right)
   &= (z,0,1).
\end{align*}
Thus, the boundary of the projection is
\begin{equation}
\partial\Pi_{M}(L) = \left\{ \left( \sqrt{\frac{1}{1+|a|^2}}e^{i\psi}, 0 \right) \in \hyp  \mid 0\leq \psi \leq 2\pi\right\}.
\end{equation}
Since \begin{equation}
\cos^2\theta = N(\nn, \mm) = \frac{1}{1+|a|^2},
\end{equation}
the projection $\Pi_{M}(L)$ is a disc  centered at $p$ with radius $\log\frac{1+|\cos\theta|}{1-|\cos\theta|}$.
\end{proof}

\subsection{Normalizations}\label{sec:nomrmalization}


\begin{figure}[htb]\label{fig:horizontal}
 \begin{center}
     \includegraphics[width=5.5cm]{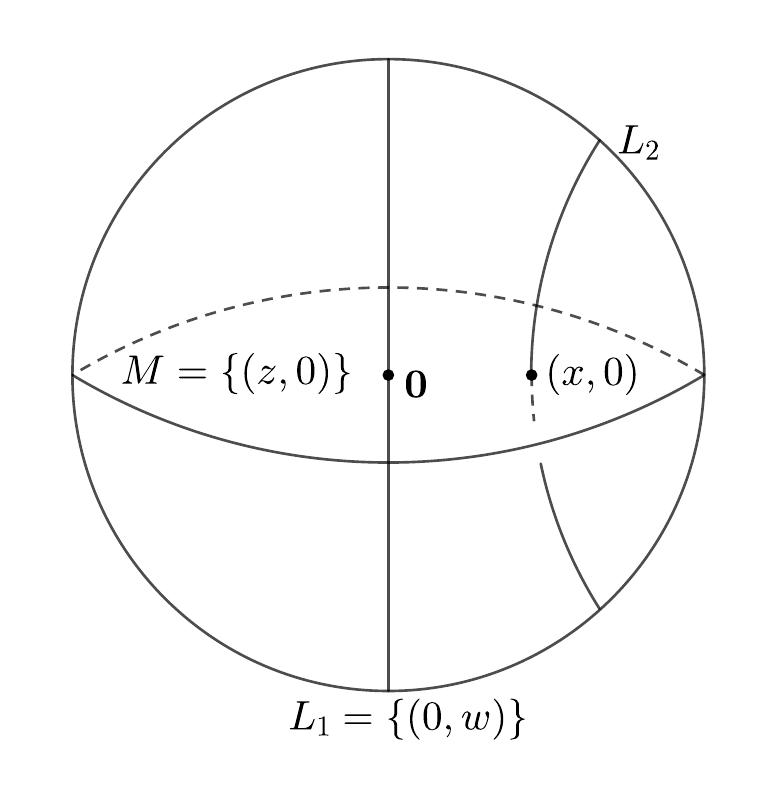}
    \caption{$d={\rho}(L_1, L_2) = {\rho}(\mathbf{0}, (x,0) ) $}
    \label{fig:setting}
  \end{center}
\end{figure}

Here, we perform some normalization of two disjoint complex lines  $L_1$ and $L_2$ with a distance $d>0$ a part which we will be using later (Figure \ref{fig:setting}).
They have a unique common perpendicular geodesic which is contained in a unique complex line, say $M$.
Since the isometry group ${\rm PU}(2,1)$ acts on the set of complex lines transitively,
we may assume that $L_1 = \{ (0,w) \in \C^2 \mid |w|<1 \}$.
Since the stabilizer ${\rm Stab}(L_1)$ also acts on $L_1$ transitively,
we may assume $L_1 \cap M = \{ \vo \}$ so that $M=\{ (z,0) \in \C^2 \mid |z| <1 \}$.
Finally, applying the rotation around $L_1$ if necessary, we may assume that the distance $d$ from $L_1$ to $L_2$ is realized by the geodesic segment connecting $\vo$ and $(x,0)$ in $M$ for $x = {\rm tanh}\frac{d}{2}$.
i.e, $M \cap L_2 = \{(x,0)\}$.
The polar vectors of $L_1$ and $L_2$ are
$\nn_1 = (1, 0, 0)$   and  $\nn_2=\left( \frac{1}{x},0,1 \right)$, respectively.
\begin{equation}\label{eqn:x-d}
\begin{split}
 d={\rho}(L_1, L_2) & = {\rho}(\mathbf{0}, (x,0) ) = \log{\frac{1+x}{1-x}}
 \Longleftrightarrow x = \tanh\frac{d}{2}
 \end{split}
\end{equation}
Note that $M = {\Pi}^{-1}_{L_1}(\vo)$.

\begin{defn}\label{def:normalied-loxo}
For  $(0, w) \in L_1$, $w=re^{i\theta}$,
consider a loxodromic isometry $g^\psi_w \in {\rm Stab}(L_1)$  which maps $(0,0)$ to $(0,w)$ and whose axis passes through $(0,0)$.
We call such $g^\psi_w$ \textit{a normalized loxodromic  isometry},
\begin{equation}\label{normalized-lox-hol}
g^\psi_{w}
  = \bpm e^{\frac{2\psi}{3}i}  & 0 & 0 \\
          0 & \frac{ e^{-\frac{\psi}{3}i} }{\sqrt{1-r^2}} & \frac{re^{i\theta}e^{-\frac{\psi}{3}i}}{\sqrt{1-r^2}} \\
          0 &  \frac{re^{-i\theta}e^{-\frac{\psi}{3}i}}{\sqrt{1-r^2}}& \frac{e^{-\frac{\psi}{3}i}}{\sqrt{1-r^2}} \\
    \epm \in {\rm SU}(2,1).
\end{equation}
This mapping has holonomy $\psi$ (see subsection \ref{sec:holonomy}. Holonomy).
\end{defn}

The normalized loxodromic  $g^\psi_w$ in the ball model is given in coordinates
\begin{equation}
\left(z_1,z_2 \right) \mapsto
\left(\frac{\sqrt{1-r^2}e^{i\psi} z_1}{re^{-i\theta}z_{2}+1},
\frac{z_2+re^{i\theta}}{re^{-i\theta}z_2 +1} \right).
\end{equation}

%


\begin{rem}\label{normalization:loxo}
Consider $\gamma \in {\rm Stab}(L_1)$ and let ${\gamma}(\vo)= (0, w) \in L_1$.
Choose the normalized loxodromic isometry $g=g_w^\psi \in {\rm Stab}(L_1)$
so that $g^{-1}{\circ}{\gamma}(\vo) = \vo$ and $hol(\gamma) =hol(g)$. 
Then $g^{-1}{\circ}\gamma|_{L_1}$ is a rotation around $\vo$ on $L_1$, say $R$.
Note that  $R|_M = id$ and $R(L_2) = L_2$.
Thus,
\begin{equation}
\rho(L_2, {\gamma}(L_2)) = \rho(L_2, g{\circ}R(L_2)) = \rho(L_2, g(L_2)).
\end{equation}
Consider the projection  $\Pi_{L_1}(L_2)$ which is a disc centered at $\vo \in L_1$.  
Then ${g}^{-1}{\circ}{\gamma}(\Pi_{L_1}(L_2)) = \Pi_{L_1}(L_2)$ and hence
\begin{equation}
\gamma(\Pi_{L_1}(L_2)) = g(\Pi_{L_1}(L_2)).
\end{equation}
Therefore, for any  $\gamma \in {\rm Stab}(L_1)$, there exists  a normalized loxodromic isometry $g=g_w^\psi \in {\rm Stab}(L_1)$ so that $\rho(L_2, {\gamma}(L_2)) = \rho(L_2, g(L_2))$  and $ \gamma(\Pi_{L_1}(L_2)) = g(\Pi_{L_1}(L_2))$.
We will also use this normalization later.
\end{rem}

\subsection{Holonomy}\label{sec:holonomy}

Let $L \subset \hyp$  be a complex line.  For any   $g \in {\rm Stab}(L)$, we define the  {\it holonomy of $g$},
  denoted $\text{hol}({g})$, to be the oriented angle of rotation about $L$.
The  {\it holonomy group}  associated to  $L$ is the set of   elements $g  \in {\rm Stab}(L)$ for which $g|_{L}= id$.
This is a circle group.
The holonomy operation  can be interpreted as a homomorphism,
$$\text{hol}:  {\rm Stab}(L) \rightarrow \mathbb{S}^{1}.$$

\begin{prop}[Properties of holonomy]\label{prop:holonomy}
Let  $f  \text{ and } g  \in {\rm Stab}(L)$. Then
\begin{enumerate}
\item $hol({g}) \in \mathbb{S}^{1}$.

\item For a normalized loxodromic  $g^\psi_z \in {\rm Stab}(L_1)$, $hol({g^\psi_z})=\psi$.

\item $hol({fg})=hol({gf})$.

\end{enumerate}
\end{prop}

\section{Holonomy and projection}\label{sec:holonomy-angle}

For this section,
let $L_1$ and $L_2$ be two disjoint complex lines in $\hyp$ which we assume are a distance $d >0$ apart.
Without loss of generality, we may assume the normalizations as in subsection \ref{sec:nomrmalization} (see Figure \ref{fig:setting}).

\begin{lem}[Trivial holonomy]\label{lem:general-no-hol}
For any $\gamma \in {\rm Stab}(L_1)$ with trivial holonomy,
\begin{equation}
 \gamma(\Pi_{L_1}(L_2)) \cap \Pi_{L_1}(L_2) = \emptyset \Longleftrightarrow  \gamma(L_2) \cap L_2 = \emptyset.
\end{equation}
\end{lem}

\begin{proof}
Set $D =\Pi_{L_1}(L_2)$. Using Remark \ref{normalization:loxo},
it suffices to prove that
$g(D) \cap D   = \emptyset  \text{  if and only if  }  g(L_2) \cap L_2 =\emptyset$
for a normalized loxodromic isometry $g=g^0_w \in {\rm Stab}(L_1)$ (see Definition \ref{def:normalied-loxo}).
Clearly, $g(L_2) \cap L_2 \neq \emptyset$ implies $g(D) \cap D \neq  \emptyset$.
To prove the converse, set $w = re^{{\theta}i}$.
note that 
Let $\nn_3$ be the polar vector of $g(L_2)$,
 $\nn_3 = g(\nn_2)=\left( \frac{1}{x}, \frac{re^{{\theta}i}}{\sqrt{1-r^2}},  \frac{1}{\sqrt{1-r^2}}\right)$.
\begin{equation}
\begin{split}
& \left< \nn_3, \nn_3 \right> = \left< g(\nn_2), g(\nn_2) \right>  = \left< \nn_2, \nn_2 \right> =\frac{1}{x^2}-1  \\
& \left< \nn_2, \nn_3 \right> = \frac{1}{x^2} - \frac{1}{\sqrt{1-r^2}}
\end{split}
\end{equation}
Using Proposition \ref{thm:polar-formula}, $g(L_2) \cap L_2  \neq \emptyset$ if and only if
\begin{align}\label{eqn:distance-L2-L3-no-hol}
 &N(L_2, g(L_2))
    = \frac{\left[ \frac{1}{x^2} - \frac{1}{\sqrt{1-r^2}} \right]^2}{\left[ \frac{1}{x^2} -1 \right]^2}
     = \left[ \frac{(\sqrt{1-r^2} - x^2) }{\sqrt{1-r^2}(1-x^2)} \right]^2 < 1.
\end{align}
This implies that
    $ r < \frac{2\sqrt{1-x^2}}{2-x^2}$ and
\begin{align}
 {\rho}\left((0,0),\left( 0,\frac{2\sqrt{1-x^2}}{2-x^2}\right)\right)
  &=2\log\frac{1+\sqrt{1-x^2}}{1-\sqrt{1-x^2}} =2s(d).
\end{align}
Thus, $0 <  r < \frac{2\sqrt{1-x^2}}{2-x^2}$ if and only if $g(D) \cap (D) \neq \emptyset$.
\end{proof}

The following lemma is about the case with non-trivial holonomy.

\begin{lem}[Non-trivial holonomy]\label{lem:general-hol}
For any loxodromic ${\gamma}\in {\rm Stab}(L_1)$ with non-trivial holonomy $\psi$,
\begin{equation} {\gamma}(L_2) \cap L_2 = \emptyset. \end{equation}
\end{lem}

\begin{proof}
Using Remark \ref{normalization:loxo},
it suffices to prove that
$  g(L_2) \cap L_2 = \emptyset$
for a normalized loxodromic isometry $g = g^\psi_w \in {\rm stab}(L_1)$ (see Definition \ref{def:normalied-loxo}).
Set $w = re^{{\theta}i}$. Let $\nn_3$ be the polar vector of $g(L_2)$,
$\nn_3 = g(\nn_2)=\left( \frac{e^{\frac{2{\psi}i}{3}}}{x}, \frac{re^{{\theta}i}e^{-\frac{\psi}{3}i} }{\sqrt{1-r^2}}, \frac{e^{-\frac{\psi}{3}i}}{\sqrt{1-r^2}} \right)$.
\begin{equation}
\begin{split}
& \left< \nn_3, \nn_3 \right> = \left< g(\nn_2), g(\nn_2) \right>  = \left< \nn_2, \nn_2 \right> =\frac{1}{x^2}-1  \\
& \left< \nn_2, \nn_3 \right> = \frac{e^{-\frac{2}{3}{\psi}i}}{x^2} - \frac{e^{\frac{\psi}{3}i}}{\sqrt{1-r^2}} = \overline{\left< \nn_3, \nn_2 \right>}
\end{split}
\end{equation}

\begin{align}
N(L_2, g(L_2))
 & = \frac{(1-r^2) -x^2\sqrt{1-r^2}(e^{-{\psi}i}+e^{{\psi}i})+ x^4}{ (1-x^2)^2(1-r^2)}  \label{eqn:hol-N} \\
 & > \frac{(1-r^2) -2x^2\sqrt{1-r^2}+ x^4}{ (1-x^2)^2(1-r^2)}\\
 & > \frac{(x^2-\sqrt{1-r^2} )^2}{(1-x^2)^2} \text{~~~since~} 1-r^2 < 1\\
 & > \frac{(x^2-1)^2}{(1-x^2)^2} = 1
\end{align}
Thus,  $g(L_2) \cap L_2 = \emptyset$.
\end{proof}

\section{Areas, volumes,  and wedges}\label{sec:vol-wedge}

In this section, we define the notion of a {\it wedge} and compute its volume.
We give with some preliminaries first. The induced metric on a complex line is a hyperbolic metric of curvature
$-1$. A disc $D_r$ of radius $r$ in a complex line has area
\begin{equation}\label{eqn:area-disc}{\rm Area}\left(D_r\right) = 4\pi\sh^2\left(\frac{r}{2}\right).\end{equation}

Now,  suppose two disjoint complex lines $L_1$ and $L_2$ are a distance
$d>0$ apart.  Recalling that the projection of $L_2$ onto $L_1$ is a disk of radius
$$s(d)= 2{\rm arcsinh}\left( \frac{1}{{\rm sinh}\frac{d}{2}} \right),$$
 a routine computation yields

\begin{lem}\label{lem: projection radius}
 The area of the disc of radius half the projection radius is
\begin{equation}
{\rm Area}\left(D_\frac{s(d)}{2}\right) = 4\pi\sh^2\left(\frac{s(d)}{4}\right)= \frac{4\pi}{e^d-1}.
\end{equation}

\end{lem}

We next derive a volume formula using Fermi coordinates.

 The volume form for the ball model in $(z_1,z_2)$-coordinates
 (see  \cite{Goldman} or \cite{Parker})  with
$z_{1}=x_{1}+iy_{1}$ and $z_{2}=x_{2}+iy_{2}$ is

\begin{equation}
{\rm dvol}
   = \frac{16}{(1-|z_1|^2-|z_2|^2)^3}\,dx_1dy_1dx_2dy_2
   \end{equation}
 and changing the $z_1$-coordinate to  polar coordinates gives

\begin{equation}\label{eq: volume form polar}
{\rm dvol}
  =\frac{16r}{(1-r^2-|z_2|^2)^3}\,drd{\theta}dx_2dy_2.
\end{equation}

\begin{prop}\label{prop:vol-general}
Let $R$ be a region in a  complex line $L$, and $G$ the  tubular neighborhood of $R$ of  width $\epsilon>0$,
Then the volume is
\begin{equation}\label{eq: tubular nbhd}
{\rm vol}(G) = 2{\pi}\left[ \cosh^4\frac{{\epsilon}}{2} -1 \right]{\rm Area}(R).
\end{equation}
\end{prop}

\begin{proof}
Note that
$$ G = \bigcup_{p\in R} \{ q \in \Pi_L^{-1}(p) \mid {\rho}(p, q) \leq \epsilon \}.$$
Without loss of generality, we may assume that $L = \{(0, z_2) \in \hyp \}$.
For $p = (0, z_2) \in L$,  observe that
$${\Pi}_{L}^{-1}(p) = \left\{ (z_1, z_2 ) \in \hyp ~\mid~ |z_1| \leq  \sqrt{1-|z_2|^2} \right\}$$
is a disc (also a complex line) with Euclidean radius $\sqrt{1-|z_2|^2}$. 
The hyperbolic distance $t$ from $(0,z_2)$ to $(z_1, z_2)$  in ${\Pi}_{L_1}^{-1}(p)$ is
\begin{equation}\label{eqn:h}
\begin{split}
t &= \log{\frac{1+\frac{|z_1|}{\sqrt{1-|z_2|^2}}}{1-\frac{|z_1|}{\sqrt{1-|z_2|^2}}}}, \\
|z_1| & =\sqrt{1-|z_2|^2}\,\tanh\frac{t}{2}.
\end{split}
\end{equation}
Thus, we write  $G$ in Euclidean coordinates.
\begin{equation}
G =  \bigcup_{ (0,z_2) \in R} \{ (z_1, z_2) \in \Pi_L^{-1}(0,z_2)  ~\mid~  |z_1| \leq f(|z_2|)\}
\end{equation}
 where $f(|z_2|)= \sqrt{1-|z_2|^2}\,{\tanh}\frac{\epsilon}{2}.$
Set $z_1 = re^{i\theta}$.

Using formula (\ref{eq: volume form polar}), the volume of $G$ is
\begin{align}
{\rm vol}(G)
     & =\int_{0}^{2\pi}d\theta\iint_R\left[\int_0^{f(|z_2|)} \frac{16r}{(1-r^2-|z_2|^2)^3}dr \right]\, dx_2dy_2 \label{eqn:angle}\\
     &= 8\pi\iint_R\left[ \cosh^4\frac{\epsilon}{2} -1 \right]\frac{1}{(1-|z_2|^2)^2}\,dx_2dy_2\\
     &=  2\pi\left[ \cosh^4\frac{\epsilon}{2} -1 \right]{\rm Area}(R).
\end{align}
\end{proof}


We next define a key ingredient of  our proof of the tubular neighborhood theorem.

\begin{defn}\label{def:wedge}
Let $D_s(p_0)$ be the disc centered at $p_0$ with radius $s>0$ in a  complex line $L$.
For $\epsilon >0$ and $\psi \in [0, 2\pi]$, we consider a sector in $\Pi^{-1}_L(p_0)$,
 $$ S(\epsilon, \psi) = \left\{ q \in \Pi^{-1}_L(p_0) ~\mid~ \rho(p_0, q) < \epsilon,  -\frac{\psi}{2} \leq {\rm Arg}\, q  \leq \frac{\psi}{2} \right\}.$$
We define a \textit{wedge} $W(s, \epsilon, \psi)$ \textit{anchored} at $D_s(p_0)$ to be
\begin{equation}
  W(s, \epsilon, \psi) = \bigcup \big\{ \, g_p(S(\epsilon, \psi)) ~\mid~ p\in D_s(p_0) \,\big\}
  \end{equation}
where
 $g_p\in {\rm Stab}(L)$ is a loxodromic isometry  with trivial holonomy satisfying $g_p(p_0) = p$ (Figure \ref{fig:wedge}).
We call $\epsilon$ the \textit{width} of the wedge.
\end{defn}

\begin{figure}[htb]
 \begin{center}
     \includegraphics[width=4cm]{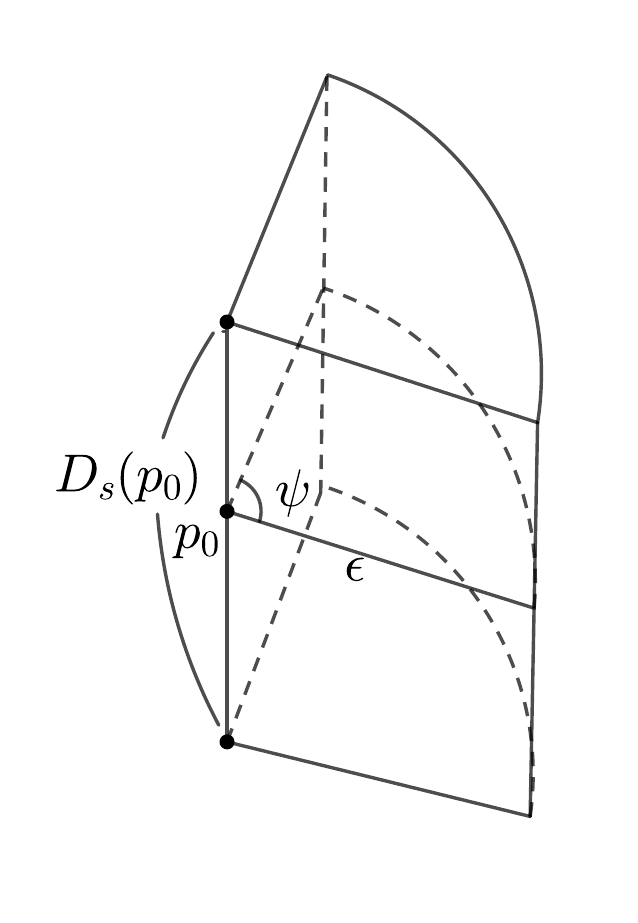}
    \caption{A wedge $W(s, \epsilon, \psi)$}
    \label{fig:wedge}
  \end{center}
\end{figure}

\begin{prop}\label{prop:vol-wedge}
For $s >0, \epsilon >0 ,  \psi \in (0, 2\pi]$,
the volume of the wedge $W(s, \epsilon, \psi)$ anchored at $D \subset L$
is
\begin{equation}
{\rm vol}(W) = \psi\left( \cosh^4\frac{\epsilon}{2} -1 \right){\rm Area}(D).
\end{equation}
\end{prop}

\begin{proof}
Two wedges are isometric if and only if they have the same wedge parameters $s, \epsilon, \psi$.
Thus, we may assume that $L = \{(0, z_2) \in \hyp \}$ and $p_0 = \vo$.

We write  the wedge $W(s, \epsilon, \psi)$ anchored at $D$ in Euclidean coordinates,
\begin{equation}
W = \bigcup_{(0,z_2) \in D} \left\{ (z_1, z_2) \in \Pi_L^{-1}(0,z_2) ~\mid~  |z_1| \leq f(|z_2|),\,  -\frac{\psi}{2} < {\rm Arg}\,z_1  < \frac{\psi}{2} \right\}
\end{equation}
 where $f(|z_2|)= \sqrt{1-|z_2|^2}\,{\tanh}\frac{\epsilon}{2}.$
Applying Proposition \ref{prop:vol-general} and
 replacing the integral  $\int_{0}^{2\pi}\,d\theta$ of equation (\ref{eqn:angle})
 by $\int_{-\frac{\psi}{2}}^{\frac{{\psi}}{2}}\,d\theta$
verifies  the proposition.
\end{proof}

\section{The Holonomy Angle Lemma and the Tubular Neighborhood Theorem}\label{sec: tubes}
The purpose of this section is to prove the tubular neighborhood theorem. An important step in proving it is the holonomy angle lemma.

 Let $X = \hyp / \Gamma $ be a complex hyperbolic manifold and
 $S \hookrightarrow X$. Suppose  $d$ is  the length of the shortest orthogeodesic from $S$ to itself. Lifting   $S$ and this orthogeodesic to
 $\hyp$, we obtain connected lifts $L_1$ and $L_2$  with $d = \rho(L_1, L_2)>0$. Recall that $\Pi_{L_1}(L_2)$ is a disc of radius $s(d)$ (Proposition \ref{prop:s}). If an element of  ${\rm Stab}_{\Gamma}(L_1)$ does not
move $\Pi_{L_1}(L_2)$  very much than not only must it have non-trivial holonomy but in fact  its holonomy angle has to be large relative to $d$.
This is made precise in the  following holonomy angle lemma.


\begin{lem}[The holonomy  angle lemma]\label{lem:holonomy}
If ${\gamma} \in   {\rm Stab}_{\Gamma}(L_1)$ is a loxodromic and
 $\gamma(\vo) \in  \Pi_{L_1}(L_2)$,
then ${\gamma}$ has non-trivial holonomy $\psi$ satisfying
\begin{equation}
  \cos{\psi}  \leq \frac{1+\tanh\frac{d}{2}}{2}.
\end{equation}
\end{lem}

\begin{proof}
Without loss of generality, we may assume the normalizations as in subsection \ref{sec:nomrmalization} (Figure \ref{fig:setting}).
Since $\gamma(\Pi_{L_1}(L_2)) \cap  \Pi_{L_1}(L_2) \neq \emptyset$, ${\gamma}$ has non-trivial holonomy by Lemma \ref{lem:general-no-hol}.
Using Remark \ref{normalization:loxo}, it suffices to prove the lemma for a normalized loxodromic (Definition \ref{def:normalied-loxo})  $g=g^\psi_w \in {\rm Stab}(L_1)$. Set $w=re^{{\theta}i}$ and $l = \rho(\mathbf{0}, (0,re^{i\theta}))$.
\begin{equation}
\text{Note that } g(\vo) \in  \Pi_{L_1}(L_2) \Longleftrightarrow l \leq s(d). \end{equation}
Let $d' =\rho(L_2, g(L_2))$ and note that $d \leq d'$. Using equation (\ref{eqn:hol-N}) we have
\begin{align}
&(1-r^2) - 2x^2\sqrt{1-r^2}\cos{\psi}+ x^4 = (1-x^2)^2(1-r^2)\cosh^2{\frac{d'}{2}}. \label{eqn:hol-x}
\end{align}
Dividing (\ref{eqn:hol-x}) by $\sqrt{1-r^2}$ and substituting the expressions
 $$x=\tanh\frac{d}{2}  \text{ and } r=\tanh\frac{l}{2}$$
 we obtain
\begin{equation}\label{eqn:distance}
{\rm sech}\frac{l}{2} -2\tanh^2{\frac{d}{2}}\cos{\psi} + \cosh\frac{l}{2}\tanh^4{ \frac{d}{2} }
    = {\rm sech}^4\frac{d}{2}{\rm sech}\frac{l}{2}\cosh^2\frac{d'}{2}.
\end{equation}
Solving for  $\cos{\psi}$
\begin{align}
  \cos{\psi}
  &= \frac{ {\rm sech}\frac{l}{2} + \cosh\frac{l}{2}\tanh^4{\frac{d}{2}}-{\rm sech}^4\frac{d}{2}{\rm sech}\frac{l}{2}\cosh^2\frac{d'}{2}}{2\tanh^2{\frac{d}{2}}}  \\
  &\leq  \frac{ {\rm sech}\frac{l}{2} + \cosh\frac{l}{2}\tanh^4{\frac{d}{2}} -{\rm sech}^2\frac{d}{2}{\rm sech}\frac{l}{2}}{2\tanh^2{\frac{d}{2}}}
      \text{~~~since~}  d \leq d'  \\
   &= \frac{1}{2}\left[{\rm sech}\frac{l}{2} + \cosh\frac{l}{2}\tanh^2{\frac{d}{2}} \right] \label{eqn:hol-func} \\
   &\leq  \frac{1}{2}\left[ 1 + \tanh{\frac{d}{2}} \right].
\end{align}
Where the last inequality follows from the  fact  that
${\rm sech}(x)  \leq 1 ,\text{~and~} l \leq s(d)$ implies $\cosh\frac{l}{2} \leq \cosh\frac{s(d)}{2}=\coth\frac{d}{2}$.
\end{proof}

Note that the cosine holonomy angle function $\frac{1}{2}\left( 1 + \tanh{\frac{d}{2}} \right)$ is increasing as a function of $d$ and
strictly bounded from above by $1$.

\begin{thm}[Tubular neighborhood]\label{thm:tubular}
There exists a positive function $c(A)$, depending only on  $A = {\rm Area}(S)$
 so that $S \hookrightarrow X$ has a tubular neighborhood $\mathcal{N}(S)$ of width at least $c(A)$.
 The function $c(A)$ can be taken to be
\begin{equation}\label{eqn:width-bound}
c(A) =\frac{1}{4}\log\left( \frac{2}{A}+1 \right).
\end{equation}
\end{thm}

\begin{rem} We remark that the surface $S$ can be a  finite area non-compact surface as long as  $S$ is not asymptotic to itself. Of course,  if the surface is closed it is automatically not asymptotic to itself. Also note that  $X$ can be  finite or infinite volume.
\end{rem}

\begin{proof}
 Suppose $d$ is the shortest orthogeodesic  from $S$ to itself and let  $L_1$ and  $L_2$  be lifts of $S$ so that $d =\rho(L_1, L_2).$
Without loss of generality, we may normalize $L_1$ and $L_2$ as in subsection \ref{sec:nomrmalization} (see Figure \ref{fig:setting}).
We next project $L_2$ into  $L_1$.
The projection $\Pi_{L_1}(L_2) \subseteq L_1$ is a disc centered at
$\vo$ with radius $s = s(d) =  2\log\coth\frac{d}{4}$ (see Equation (\ref{formula:radius-projection})), and
let $D_{\frac{s}{2}} \subseteq L_1$ be the disc with the same center as
$\Pi_{L_1}(L_2)$  but half of its radius.
Set $F$ to be the following subset of the stabilizer of $L_1$ in $\Gamma$,
 $$F=\big\{ \gamma \in {\rm Stab}_{\Gamma}(L_1) - \{ id \big\} \mid {\gamma}(D_{\frac{s}{2}}) \cap D_{\frac{s}{2}} \neq \emptyset\}.$$
Then the discreteness of $\Gamma$ implies that $F$ is a finite set,
and by Lemma \ref{lem:general-no-hol} any $\gamma \in F$ has  non-trivial holonomy.
Let $\Psi$ be the minimum holonomy of the elements of $F$,
$$\Psi = \text{min}\{ hol(\gamma) \mid  \gamma \in F \}.$$
Now, if $F \ne \emptyset$ then by definition $\Psi$ is between $0$ and $\pi$.
Otherwise, if $F=\emptyset$,  and we set $\Psi = 2\pi$.

Let  $W$ be the wedge $W(\frac{s}{2}, \frac{d}{2}, \Psi )$ anchored  on  $D_{\frac{s}{2}}$.
By construction $\gamma (W) \cap W = \emptyset$ for all $\gamma \in  {\rm Stab}_\Gamma(L_1) -\{id\}$.
Furthermore, since the width of $W$ is  $\frac{d}{2}$ and $d$ is the shortest distance between the lifts of $S$, we  have that $\gamma(W) \cap W =\emptyset$ for all $\gamma \in  \Gamma -\{id\}$.
Thus the wedge $W$ embeds in $X$ and we have,
\begin{equation}\label{eqn:v-c}{\rm vol}(W) \leq {\rm vol}(\mathcal{N}_{\frac{d}{2}}(S))\end{equation}
where $\mathcal{N}_{\frac{d}{2}}(S)$ is the tubular neighborhood of $S$ having  width $\frac{d}{2}$.
Using Propositions \ref{prop:vol-general} and \ref{prop:vol-wedge} for the volume on both sides of (\ref{eqn:v-c}) we have
\begin{equation}\label{thm-pf-1}
 \Psi \left[ \cosh^4\frac{d}{4} -1 \right]{\rm Area}\left(D_{\frac{s}{2}}\right) <  2\pi \left[ \cosh^4\frac{d}{4} -1 \right]{\rm Area}(S).
\end{equation}

Using the fact that $ {\rm Area}\left(D_\frac{s}{2}\right) = 4\pi\sh^2\left(\frac{s(d)}{4}\right)= \frac{4\pi}{e^d-1}$
(Lemma \ref{lem: projection radius}) and
solving for $\Psi$ we obtain
\begin{equation}\label{eq:inequality for psi}
  \Psi < \frac{A}{2}(e^d -1).
\end{equation}
There are two cases to consider.\\
Case (1).  $\frac{\pi}{2} \leq \Psi \leq \pi$:
 inequality (\ref{eq:inequality for psi}) implies that  $\frac{\pi}{2}  \leq \Psi < \frac{A}{2}  (e^d -1)$. Solving for $d$ we have,
\begin{equation} \label{eq:lower bound d, case 1}
d > \log\left( \frac{\pi}{A}+1 \right).
\end{equation}
Case (2). $0 \leq \Psi \leq \frac{\pi}{2}$:
an application of the holonomy angle Lemma (Lemma \ref{lem:holonomy}) coupled with the inequality  (\ref{eq:inequality for psi}) yields
\begin{equation}\label{eq:width-estimation}
\arccos\left( \frac{1+ \tanh\frac{d}{2}}{2} \right)
\leq  \frac{A}{2}(e^{d}-1).
\end{equation}
Since $\arccos{x} \geq 1-x $ for $0 \leq x \leq 1$, we obtain
\begin{equation}
1- \left[\frac{1+\tanh\frac{d}{2}}{2} \right]  \leq \frac{A}{2} (e^d-1).
\end{equation}
Solving for $d$, we have
\begin{equation} \label{eq:lower bound d, case 2}
d \geq \frac{1}{2} \log\left(\frac{2}{A}+1 \right).
\end{equation}
Considering  the two cases (\ref{eq:lower bound d, case 1}) and (\ref{eq:lower bound d, case 2}),  the distance $d = \rho(L_1, L_2)$ must satisfy
\begin{equation*}
d \geq \min \bigg\{ \log\left( \frac{\pi}{A}+1 \right), \frac{1}{2} \log\left( \frac{2}{A}+1 \right) \bigg\}
  =\frac{1}{2} \log\left( \frac{2}{A}+1 \right).
\end{equation*}
Therefore $S$ has a tubular neighborhood of width
\begin{equation*}
\frac{1}{4}\log\left( \frac{2}{A}+1 \right).
\end{equation*}
\end{proof}

\section{Disjoint  embedded surfaces}\label{sec: Two embedded surfaces}

Let $X = \hyp / \Gamma $ be a complex hyperbolic manifold and $S_1, S_2 \hookrightarrow X$ two disjoint embedded  complex geodesic surfaces in $X$ with distance $d = {\rho}(S_1, S_2) > 0$.
Lifting the shortest orthogeodesic from $S_1$ to $S_2$ to the universal cover $\hyp$,
we have two complex lines $L_1$ and $L_2$ with distance $\rho(L_1, L_2)=d$.
Without loss of generality, we normalize $L_1$ and $L_2$ as in subsection \ref{sec:nomrmalization} (see Figure \ref{fig:setting}).
Note that the projection $\Pi_{L_1}(L_2) \subseteq L_1$ is a disc
$D_{s}$ centered at
$\vo$ with radius $s = s(d) =  2\log\coth\frac{d}{4}$ (Equation (\ref{formula:radius-projection})).
Denote  the disc with the same center as $\Pi_{L_1}(L_2)$  but half of its radius by $D_{\frac{s}{2}} \subseteq L_1$ and $A_k$  the area of the surface $S_k$ ($k=1,2$).

\begin{lem}\label{lem:two-embedded}
If  $\gamma(\vo) \in  \Pi_{L_1}(L_2)$ for some non-identity element
${\gamma} \in   {\rm Stab}_{\Gamma}(L_1)$,
then
\begin{equation} \label{eq: distance from L1 to L2}
 \rho(L_1, L_2)> \frac{1}{4} \log \left(\frac{2}{{\rm max}\{A_{1},A_{2}\} }+1\right).
\end{equation}
\end{lem}

\begin{proof}
Set $d=\rho(L_1, L_2)$ and
$d' = \rho(L_2, {\gamma}(L_2))$. There will be finitely many translates
of the origin that are contained in the disc  $\Pi_{L_1}(L_2)$. Of course,
by Lemma  \ref{lem:general-no-hol},  they must all have non-trivial holonomy. Let $\gamma$ be the one with the smallest holonomy.
Using the normalization of Remark \ref{normalization:loxo}, we may assume that $\gamma$ is the normalized loxodromic.
Set $l = \rho(\mathbf{0}, \gamma(\mathbf{0}))$.
From  the holonomy angle Lemma \ref{lem:holonomy} and Equation (\ref{eqn:distance}) we have
\begin{align}
\cosh^2\frac{d'}{2}
 &= \cosh^4\frac{d}{2} - 2\cosh\frac{l}{2}\cosh^2\frac{d}{2}\sinh^2\frac{d}{2}\cos\Psi+\cosh^2\frac{l}{2}\sinh^4\frac{d}{2}.
\end{align}
There are two cases to consider depending on the sign of $\cos\Psi$.

\begin{itemize}

\item  $\cos\Psi \geq 0$:

Applying the fact that $l \leq s(d)$ implies $\cosh\frac{l}{2}\leq \coth\frac{d}{2}$,
\begin{align}
\cosh^2\frac{d'}{2}
  &\leq \cosh^2\frac{d}{2}\left[\cosh^2\frac{d}{2}+\sinh^2\frac{d}{2}\right] \label{eqn:2-em-1}\\
  &\leq \left[\cosh^2\frac{d}{2}+\sinh^2\frac{d}{2}\right]^2 =\cosh^2d.
\end{align}
Thus, we obtain $\frac{d'}{2} \leq d$ or equivalently,
$$\rho(L_1,L_2)>\frac{1}{2}\rho(L_2, {\gamma}(L_2))>
\frac{1}{4} \log \left(\frac{2}{A_2}+1\right)
$$
where the last inequality follows from Theorem \ref{thm:tubular}.

\item  $\cos\Psi \leq 0$ or equivalently $\frac{\pi}{2} \leq \Psi \leq \pi$.
Set $\epsilon =\frac{1}{4} \log \left(\frac{2}{A_1}+1 \right)$.

First observe that the wedge

\begin{align*}
W=  W\left(\frac{s(d)}{2}, \epsilon, \frac{\pi}{2} \right)
\end{align*}
embeds in the tubular neighborhood $\mathcal{N}_{\epsilon}(S_{1}).$
This follows from the fact that any element of ${\rm Stab}_{\Gamma}(L_1)$ either moves the disc $\Pi_{L_1}(L_2)$ away from itself or if it doesn't then has holonomy bigger than $\frac{\pi}{2}$.
Next we compare the volume of $W$ with the volume of
$\mathcal{N}_{\epsilon}(S_{1})$.  Using expresion (\ref{thm-pf-1})  we have
\begin{equation}
 \frac{\pi}{2} \left[ \cosh^4\frac{\epsilon}{2} -1 \right]{\rm Area}\left(D_{\frac{s(d)}{2}}\right) <  2\pi \left[ \cosh^4\frac{\epsilon}{2} -1 \right] A_{1}.
 \end{equation}
 Using Lemma \ref{lem: projection radius} and simplifying we have

\begin{equation}
 \frac{\pi}{e^d-1} < A_{1}
\end{equation}
or equivalently
$$ \rho(L_1, L_2)> \log \left(\frac{\pi}{A_1}+1 \right)
>\frac{1}{4}\log \left(\frac{2}{A_1}+1 \right).
$$
\end{itemize}
In either case we have
$$\rho(L_1, L_2)>{\rm min}
 \left\{\frac{1}{4}\log \left(\frac{2}{A_1}+1\right), \frac{1}{4}\log
 \left(\frac{2}{A_2}+1\right)\right\}
 $$
 and the conclusion follows.
\end{proof}

\begin{lem} \label{lem: translates disjoint}
If  $\gamma(\vo) \notin  \Pi_{L_1}(L_2)$ for all  non-identity elements
${\gamma} \in   {\rm Stab}_{\Gamma}(L_1)$,
then
\begin{equation}
 \rho(L_1, L_2)> \log \left(\frac{4\pi}{A_1}+1\right).
\end{equation}
\end{lem}

\begin{proof}
Recall $d=\rho (L_1, L_2)$.
Since $\gamma(\vo) \notin  \Pi_{L_1}(L_2)$ for all non-identity elements
$\gamma \in  {\rm Stab}_{\Gamma}(L_1)$, the disc of radius
$\frac{s(d)}{2}$ must embed in the quotient surface $S_1$. Using
Lemma \ref{lem: projection radius} we have
${\rm Area}\left(D_\frac{s(d)}{2}\right) <A_1$, and hence
 \begin{equation}
 4\pi\sh^2\left(\frac{s(d)}{4}\right)= \frac{4\pi}{e^d-1}<A_1
\end{equation}
Solving for $d$ we have
$d> \log \left(\frac{4\pi}{A_1}+1\right).$
\end{proof}

\begin{thm}\label{thm:2-embedded}
Let $X = \hyp / \Gamma $  and $S_1, S_2 \hookrightarrow X$ two disjoint embedded complex geodesic surfaces in $X$ with distance ${\rho}(S_1, S_2) >0.$
Then $S_1$ and $S_2$ have disjoint tubular neighborhoods of width
$$
\frac{1}{8}
 \log \left(\frac{2}{{\rm max}\{A_1, A_2\}}+1\right)
$$
where $A_k = {\rm Area}(S_k)$ for $k=1,2$.
\end{thm}

\begin{proof}
Putting Lemmas \ref{lem:two-embedded} and \ref{lem: translates disjoint} together and noting that by assumption $\rho(L_1,L_2)=\rho(S_1, S_2)$, we see that
\begin{displaymath}
 {\rho}(S_1, S_2) >
 {\rm min}
 \left\{\frac{1}{4} \log \left(\frac{2}{{\rm max}\{A_{1},A_{2}\} }+1\right),
\log \left(\frac{4\pi}{A_1}+1\right) \right\}
>\frac{1}{4} \log \left(\frac{2}{{\rm max}\{A_{1},A_{2}\} }+1\right)
  \end{displaymath}

 Moreover, we know from  Theorem  \ref{thm:tubular}
 that $S_1$ and $S_2$ have tubular neighborhoods of width
$\frac{1}{4} \log \left(\frac{2}{A_1}+1\right)$ and
$\frac{1}{4} \log \left(\frac{2}{A_2}+1\right)$, respectively.
 It follows that  a  width of
\begin{displaymath}
 \frac{1}{8}
 \log \left(\frac{2}{{\rm max}\{A_1, A_2\}}+1\right)
\end{displaymath}
is a tubular neighborhood for both $S_1$ and $S_2$, and moreover they are disjoint.
\end{proof}

\begin{cor}\label{cor:main}
Suppose that $X = \hyp / \Gamma $ is a complex hyperbolic manifold and
 $S_1, S_2, \dots ,S_n$ are pairwise disjoint embedded complex geodesic surfaces in $X$ satisfying ${\rho_{X}}(S_k, S_j)>0$, for $k\neq j$.
Then \begin{equation} \label{cor: lower volume bound}
{\rm vol}(X) \geq  2{\pi}nA\Biggl[ \cosh^4 \bigg( {\frac{1}{16}\log\left(\frac{2}{A} +1 \right)}\biggr) - 1 \Biggr]
\end{equation}
where $A = {\rm max}\{ {\rm Area}(S_1),  \dots ,{\rm Area}(S_n) \}$.
\end{cor}

\begin{proof}
Since the smallest of the tubular widths about  $S_1,  \dots ,S_n$
occurs for the largest area surface,  we can apply Theorem \ref{thm:2-embedded} to guarantee  that
$S_1, S_2, \dots ,S_n$ have   disjoint embedded tubular neighborhoods of width $
 \epsilon =\frac{1}{8}
 \log \left(\frac{2}{A}+1\right)$.  Hence
$$
{\rm vol}(X) \geq n\,{\rm vol}\left(\mathcal{N}_{\epsilon}(S)\right)
$$
where $S$ is the surface of (largest) area $A$.
Finally using the volume formula,  equation (\ref{eq: tubular nbhd}),  with
$\epsilon=\frac{1}{8}\log\left( \frac{2}{A}+1 \right)$
yields inequality  (\ref{cor: lower volume bound}).
\end{proof}

\section{The Eigenvalue problem} \label{sec: eigenvalue}

We use \cite{Buser-book, Cha, Cha2} as  basic references for this section. Our tubular neighborhood theorem can be used to give  geometric bounds on the first eigenvalue of the Laplacian.
We illustrate this principle   by giving  a bound on the first non-zero eigenvalue of the closed eigenvalue problem.

Let $X$ be a closed Riemannian manifold. The goal of the closed eigenvalue problem is to find  real numbers $\lambda$ for which there is a non-trivial
$C^{2}$ solution to $\Delta \phi + \lambda \phi =0$. The set of eigenvalues
are nonnegative, discrete and go to infinity.
In what follows $\lambda(X)$ denotes the first non-zero eigenvalue of the closed eigenvalue problem on the complex hyperbolic manifold $X$.
If $S \hookrightarrow X$ is an embedded complex geodesic surface of Euler characteristic $\chi$,  denote its  tubular neighborhood of width
$c(|\chi|)$ by
$\mathcal{N}_{c(|\chi|)}$.

\begin{thm} \label{thm: eigenvalue upper bound}
Suppose we have an embedded  compact complex geodesic  surface
$S \hookrightarrow X$, where $X$ is a closed  complex hyperbolic 2-manifold. Then
\begin{equation} \label{eq: upper bound}
\lambda (X) \leq \frac{{\rm vol}(\mathcal{N}_{c(|\chi|)})}{(c(|\chi |))^{2}[{\rm vol}(X)-{\rm vol}(\mathcal{N}_{c(|\chi|)})]}.
\end{equation}

\end{thm}
\begin{proof}  The tubular neighborhood theorem (Theorem \ref{thm:tubular})  guarantees that $S$ has a tubular neighborhood
 $\mathcal{N}_{c(|\chi|)}$ of width $c(|\chi |)$ where $\chi$ is the Euler characteristic  of $S$.
 Rayleigh's theorem gives a variational characterization of the first eigenvalue.  In particular,
\begin{equation}\label{eq: Rayleigh}
\lambda (X) \leq \frac{\int_{X} |\text{grad}(f)|^{2}\, {\rm dvol}}{\int_{X} f^{2} \,\rm{dvol}}
\end{equation}
where $f$ is  a  test function in the space of (Sobolev) metric completion of the smooth functions on $X$ (see \cite{Cha2} for more background and details). Define

\begin{equation}
f(x)=
\begin{cases}
      d(x,S), & \text{if } x \in \mathcal{N}_{c(|\chi|)}\\
       c(|\chi |), & \text{if } x  \notin  \mathcal{N}_{c(|\chi|)}
    \end{cases}
\end{equation}

For the numerator of the Rayleigh  quotient  (\ref{eq: Rayleigh}), note that $|\text{grad}(f)|=1$ on $\mathcal{N}_{c(|\chi|)}$,  and is $0$ off
$\mathcal{N}_{c(|\chi|)}$.
Hence
\begin{equation}\label{eq: gradient}
\int_{X} |\text{grad}(f)|^{2}\, {\rm dvol} ={\rm vol}(\mathcal{N}_{c(|\chi|)}).
\end{equation}
On the other hand, the denominator of (\ref{eq: Rayleigh}) has lower bound given by
\begin{equation}\label{eq: norm lower bound}
\int_{X} f^{2} \, {\rm dvol} \geq \int_{X-\mathcal{N}_{c(|\chi|)}}  (c(|\chi |))^{2} \, {\rm dvol}
=(c(|\chi |))^{2}[{\rm vol}(X)-{\rm vol}(\mathcal{N}_{c(|\chi|)})].
\end{equation}
Putting  (\ref{eq: gradient})  and  (\ref{eq: norm lower bound})  into  the
Rayleigh quotient
yields the desired upper bound on $\lambda(X)$.
\end{proof}

Using Theorem \ref{thm: eigenvalue upper bound}  we obtain an upper bound on $\lambda (X)$ in terms of the
Euler characteristic of $S$ and the volume of $X$.

\begin{cor} \label{cor: explicit estimate}
Same hypothesis as Theorem
\ref{thm: eigenvalue upper bound}. Then
\begin{displaymath}
\lambda (X) \leq
 \frac{64{\pi}^{2}|\chi |\left[ \cosh^4 \left(\frac{1}{8}\log\left( \frac{1}{\pi |\chi |}+1\right)\right)  -1
\right]}{\left(\log\left( \frac{1}{\pi |\chi |}+1 \right)\right)^{2}
\left({\rm vol}(X)-4{\pi}^{2}|\chi |
\left[ \cosh^4 \left(\frac{1}{8}\log\left( \frac{1}{\pi |\chi |}+1\right)\right)  -1
\right]\right)}.
\end{displaymath}
\end{cor}

\begin{proof}
If we substitute  formula (\ref{eqn:width-bound}) for the value of
$c(|\chi |)$, and substitute formula (\ref{eq: tubular nbhd})  with
$\epsilon=c(|\chi |)$,  for the value of
${\rm vol}(\mathcal{N}_{c(|\chi|)})$ then  the upper bound
in (\ref{eq: upper bound})  becomes

\begin{equation} \label{eq: upper bound2}
\frac{4{\pi}^{2}|\chi |\left[ \cosh^4 \left(\frac{1}{8}\log\left( \frac{1}{\pi |\chi |}+1\right)\right)  -1
\right]}{\left( \frac{1}{4}\log\left( \frac{1}{\pi |\chi |}+1 \right)\right)^{2}
\left({\rm vol}(X)-4{\pi}^{2}|\chi |
\left[ \cosh^4 \left(\frac{1}{8}\log\left( \frac{1}{\pi |\chi |}+1\right)\right)  -1
\right]\right)}.
\end{equation}

\end{proof}

\section{Examples}\label{sec: examples}

This section is devoted to examples.
In  \ref{subsec: no collar function}  and \ref{subsec: no collar function in cx hyp} we show that  in the setting of simple closed geodesics in hyperbolic 3-manifolds or in  complex hyperbolic manifolds  there is no positive  function $F(x)$ for which any simple closed geodesic of length $\ell$ has a tubular neighborhood of width
$F(\ell)$.
In \ref{subsec: examples of cx surfaces} we mention some references to the existence of complex geodesic surfaces in complex hyperbolic manifolds.

\subsection{Example (No collar  function for geodesics in a $3$-manifold)}
\label{subsec: no collar function}
The Margulis lemma guarantees a tube of a certain size if the  simple closed geodesic is short \cite{Brooks-M}. In this set of examples we show that  if we relax the shortness condition, the statement fails in a  strong  way. Following  \cite{Bas-Wolpert},  a Riemannian manifold is said to have the  {\it spd-property} if all of its primitive closed geodesics are
simple and pairwise disjoint.

We start with a compact topological $3$-manifold $M$ that is either a handlebody
$V_{g}$, ($g \geq 2$) or $\Sigma  \times \mathbb{R}$, where
$\Sigma=\Sigma_{g,n}$  is  a surface with negative Euler characteristic.
For $\rho$  a marked  hyperbolic structure on $M$,  we denote $M$ with this metric by $M_{\rho}$,  and the $\rho$-length  of a geodesic $\gamma$ in
$M_{\rho}$ by $\ell_{\rho}(\gamma)$.

Fix homotopy classes  of non-simple closed curves $\{\gamma_{1}, \cdots ,\gamma_{m}\}$  on one of the boundary components of $M$, and fix a marked Fuchsian hyperbolic structure $\rho$ on $M$.

\begin{prop}\label{prop: no tube theorem}
Given $\epsilon >0$, there exists a complete hyperbolic metric
$\rho_{\epsilon}$ on $M$ where
\begin{itemize}
\item $(M, \rho_{\epsilon})$  has the $\text{spd}$-property.  In particular, the geodesic in the homotopy class of $\gamma_{i}$ is simple
\item
 $\ell_{\rho}(\gamma_{i})-\epsilon \leq \ell_{\rho_{\epsilon}}(\gamma_{i}) \leq \, \ell_{\rho}(\gamma_{i}) +\epsilon$, for $i=1,...,m$
\item $\ell_{\rho_{\epsilon}} (\delta_{ij}) < \epsilon$, where $\delta_{ij}$ is the shortest orthogeodesic from $\gamma_{i}$ to $\gamma_{j}$.
\end{itemize}
\end{prop}

\begin{proof}
We outline the proof. View the metric $\rho$ as a discrete faithful Fuchsian representation of  Schottky space
 if $M$ is a handlebody, and of   quasifuchsian space if $M$ is $\Sigma  \times \mathbb{R}$.
We use the convention here that if a closed geodesic self-intersects itself  then we set the length of the orthogeodesic to be zero.
The length function (of a curve or orthogeodesic)  is a continuous function on either of these representation spaces. Since the
$\{\gamma_{i}\}_{i=1}^{m}$ are non-simple  closed
geodesics on $M_{\rho}$ and hence $\ell_{\rho}(\delta_{ij})=0$,
we can find a connected open set  $U$
containing $\rho$ so that  the lengths
$\ell_{\rho^{\prime}} (\delta_{ij}) < \epsilon$ for any
$\rho^{\prime} \in U$.   Moreover, by choosing $U$ small enough, the continuity guarantees that
\begin{equation}
\ell_{\rho}(\gamma_{i})-\epsilon \leq \ell_{\rho_{\epsilon}}(\gamma_{i}) \leq \, \ell_{\rho}(\gamma_{i}) +\epsilon  \text{ for } i=1,...,m.
\end{equation}
Finally, it follows from  \cite{Bas-Wolpert} that
there is a dense subset of representations in the open set $U$ that have the spd-property.
In particular, a  representation $\rho_{\epsilon}$ exists in $U$  satisfying
the  three items.
\end{proof}

\subsection{Example (No collar  function for geodesics in a complex hyperbolic $2$-manifold)}\label{subsec: no collar function in cx hyp}
We outline a Schottky construction of an  elementary example that shows that two simple closed geodesics of  fixed lengths  in a complex hyperbolic $2$-manifold   can be  arbitrarily close to each other.  Hence there is no width that guarantees a tubular neighborhood.

We work in the ball model
\begin{equation}
\mathbb{H}^2_{\mathbb{C}}  =  \{ (z_1, z_2)\in\mathbb{C}^2 \mid |z_1|^2 + |z_2|^2 < 1 \}.
\end{equation}

Consider the oriented complete geodesic $\ell_1$ with endpoints at infinity
going from $(-1,0)$ to $(1,0)$.   Also consider a second oriented complete geodesic $\ell_{2}$ with endpoints going from $(0,-1)$ to $(0,1)$.
Note that $\ell_1$ and $\ell_2$ intersect at the origin.
Choose neighborhoods $U_k \subset \partial \mathbb{H}^2_{\mathbb{C}}$ and $V_k \subset \partial \mathbb{H}^2_{\mathbb{C}}$ about the endpoints of $\ell_k$, for $k=1,2$.
Construct  a hyperbolic element $\gamma_k$ whose axis is $\ell_k$, for  $k=1,2$ having  translation length long enough so that
 $\gamma_k(\text{ext}(U_k))\subset V_{k}$ for $k=1,2$.
Now the group $<\gamma_1, \gamma_2>$ acts discretely on
 $\mathbb{H}^2_{\mathbb{C}}$ and is isomorphic to a  free group on two generators. Now by moving the axis of $\gamma_2$ slightly off of
 the axis of $\gamma_1$ but not changing translation lengths we still have a  discrete subgroup isomorphic to a free product. The quotient
 $\mathbb{H}^2_{\mathbb{C}}/<\gamma_1, \gamma_2>$  has two simple closed geodesics that are disjoint that  can be made arbitrarily close to each other. Hence neither closed geodesic has a tubular neighborhood.


\subsection{Examples of complex geodesic surfaces in a complex hyperbolic $2$-manifold}\label{subsec: examples of cx surfaces}

For the existence of infinitely many   lattices  which contain an embedded subgroup which keeps invariant a
complex geodesic see \cite{Borel-Harish} or \cite{Chinburg-Stover}.
These lattices are in fact arithmetic of simple type.
For example, ${\rm PU}(n,1, \mathcal{O}_{d}) < {\rm PU}(n,1)$, where
$ \mathcal{O}_{d}$ is the ring of integers of the imaginary quadratic field
$\mathbb{Q}(\sqrt{-d})$.


\section{A  Combination Theorem and Bounds on the  Tube Function}
\label{combination thm and tube bounds}
The goal of this section is to give an asymptotic estimate for the optimal tube (collar) function of an embedded complex geodesic surface.
 Along the way, as a tool we prove a geometric version of the combination theorem which may be of independent interest. Much of this section follows the same approach of sections  5 and 6   in \cite{B-tubular} for finding the optimal collar function in the setting of totally geodesic hypersurfaces in  real hyperbolic manifolds. The two  significant  differences  are  that  bisectors in complex hyperbolic space are not totally geodesic and we are not in real codimension one. Fortunately the differential geometry of bisectors  in complex hyperbolic space has been carefully studied. We refer to Goldman  \cite {Goldman}  for the relevant properties.  For the basics on combination theorems we refer to Maskit \cite{Maskit}.

\subsection{A geometric combination theorem}\label{subsec: comb thm}

A pair $(\Gamma_{1}, L_{1})$ is said to be a $\mathbb{C}$-Fuchsian group, if $\Gamma_{1} < {\rm PU}(2,1)$ is discrete and keeps invariant the complex line $L_{1}$.
For a $\mathbb{C}$-Fuchsian group   $(\Gamma_{1}, L_{1})$, and  $z \in L_1$,  we define  ${\rm inj}(z_1)$ to be the {\it injectivity radius}  of the projection of  $z_{1}$ in the surface  $L_{1}/\Gamma_{1}$.

\begin{figure}[htb]
 \begin{center}
     \includegraphics[width=7cm]{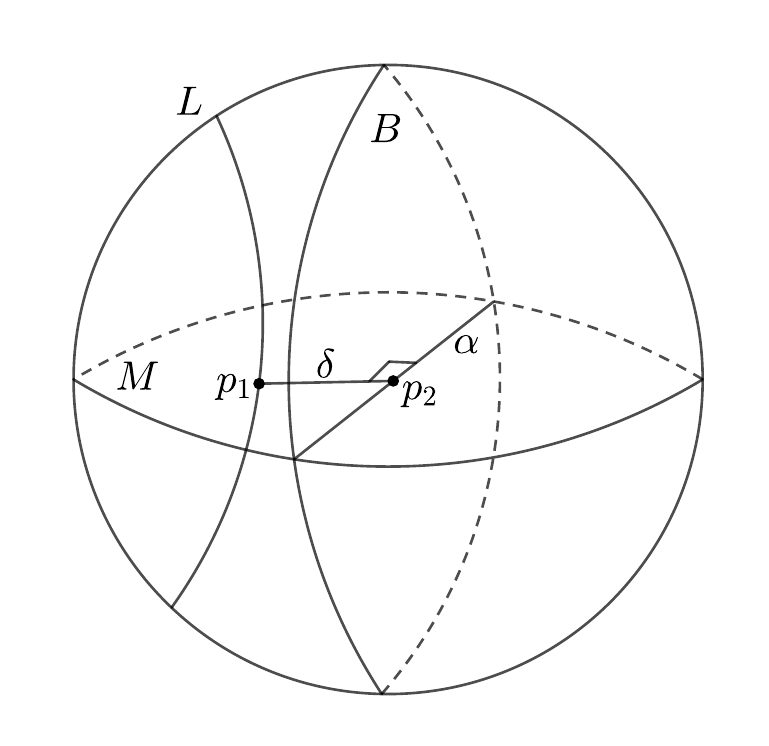}
    \caption{A complex line $L$ and a bisector $B$}
    \label{fig:bisector-proj}
  \end{center}
\end{figure}

\begin{lem} \label{lem: bisector properties}
 Let $L$ be a complex line, $\alpha$ a geodesic with $\rho(\alpha, L) = d>0$, and $B$ a bisector whose spine is $\alpha$.
Suppose the unique complex line $M$ containing $\alpha$ intersects $L$ transversally on a point. Then
\begin{enumerate}
\item   $\Pi_{L}(B)$ is a disc of radius $s(d)$.
\item  $\gamma\Pi_{L}(B) \cap \Pi_{L}(B)=\emptyset$ implies $\gamma(B) \cap B=\emptyset$, for $\gamma \in {\rm Stab}(L)$
\item The boundary of a bisector $B$ in ${\partial}\mathbb{H}_{\mathbb{C}}^2$  bounds
two open connected sets.
\end{enumerate}
\end{lem}
\begin{proof}
To prove item (1), note that the unique orthogeodesic from $L$ to $\alpha$, say $\delta$, is contained in  $M$.
Let $p_1, p_2 \in M$ be the end points of $\delta$, i.e, $L \cap \delta =\{ p_1 \}$ and $\alpha \cap \delta = \{ p_2 \}$.
Then note that $L = \Pi^{-1}_M(p_1)$ and $B = \{ \Pi^{-1}_M(q) : q \in \alpha \}$  where $\Pi_M : \hyp \rightarrow M$ is  the orthogonal projection to $M$.
Without  loss of generality, we may assume that $M=\{(z,0)\in \hyp : |z| < 1 \}$, $p_1=(x,0)$  for $x = \tanh\frac{d}{2}$, and $p_2=(0,0)$  (Figure \ref{fig:bisector-proj}).
Under our normalization, $\alpha = \{ (yi, 0) \in \hyp : |y|< 1\}$.
Since for any $q \in \alpha$, the following inequality $$\rho(L, \Pi^{-1}_M(q)) =\rho(p_1, q) \geq \rho(p_1, p_2) = d$$
implies that the projection $\Pi_{L}(\Pi^{-1}_M(q))$ is a disc centered at $p_1$ of radius  smaller than or equal to $\Pi_{L}(\Pi^{-1}_M(p_2))$.
Thus, $\Pi_{L}(B) = \Pi_{L}(\Pi^{-1}_M(p_2))$.
By Lemma \ref{prop:s}, the projection $\Pi_{L}(\Pi^{-1}_M(p_2)$ is a disc centered at $p_1$ of radius $s(d)$ in $L$.
Therefore $\Pi_{L}(B)$ is a disc of radius $s(d)$.

Item (2) follows from the fact that $\Pi_{L}(B) = \Pi_{L}(\Pi^{-1}_M(p_2))$ and Lemma \ref{lem:general-no-hol}.

Finally,   Item (3) follows from the fact that  the boundary of the bisector   $B$ is a smooth $2$-sphere in
$\partial \mathbb{H}_{\mathbb{C}}^2$, and therefore by the generalized
Jordan curve theorem for smooth embedded spheres \cite{Lima},
$\partial B$ divides $\partial \mathbb{H}_{\mathbb{C}}^2$ into two
open connected sets.
\end{proof}

\begin{thm}\label{thm: combination theorem}
Let $(\Gamma_{1}, L_{1})$ and $(\Gamma_{2}, L_{2})$ be
$\mathbb{C}$-Fuchsian groups where $L_1$ and $L_2$ are disjoint, and
let $p_1 \in L_1$ and $p_2 \in L_2$ be the endpoints of the unique orthogeodesic $\delta$ from $L_1$ to $L_2$.
Suppose
\begin{equation}\label{eq: injectivity inequality}
s({\rm inj}(p_1))+s({\rm inj}(p_2)) < \rho (L_1,L_2).
\end{equation}
Then  $\Gamma =<\Gamma_1,\Gamma_2>$  is a discrete subgroup of
${\rm PU}(2,1)$ where
\begin{enumerate}
\item  $\Gamma$  is abstractly  the free
product $\Gamma_1 \ast \Gamma_2$
\item $L_{1}/\Gamma_{1}$ and $L_{2}/\Gamma_{2}$ are embedded complex geodesic surfaces in $X=\mathbb{H}_{\mathbb{C}}^2 /\Gamma$
\item  $\rho_{X} (S_{1},S_{2})=\rho (L_1,L_2).$
\end{enumerate}
\end{thm}

\begin{figure}[htb]
 \begin{center}
     \includegraphics[width=8cm]{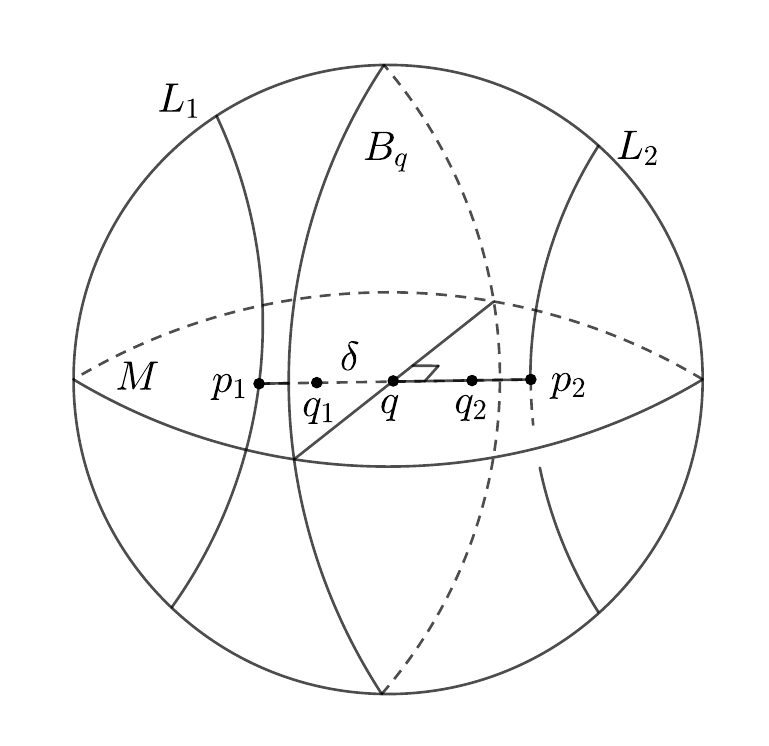}
    \caption{Bisector $B_q$}
    \label{fig:bisector}
  \end{center}
\end{figure}

\begin{proof}
The orthogeodesic $\delta$ from $p_1$ to $p_2$  defines a one parameter
family of disjoint bisectors, $\{B_{p} : p \in \delta\}$ in the following way.
The orthogeodesic $\delta$ is contained in the unique complex line $M$ containing $p_1$ and $p_2$.
For $p \in \delta \subset \mathbb{H}_{\mathbb{C}}^2$
 there is a unique geodesic in $M$
perpendicular to $\delta$ at $p$. Denote this geodesic by
$\alpha$. Then the bisector at $p$ is
\begin{equation*}
B_{p} =\{\Pi^{-1}_{M}(q): \, q \in \alpha \}.
\end{equation*}
Note that $B_{p}$ is foliated by complex lines.

Our goal is to  apply the combination theorem with respect to the action
of $\Gamma_1$ and $\Gamma_2$ on $\partial \mathbb{H}_{\mathbb{C}}^2$.
Let $q_k \in \delta$ be the point a distance $s({\rm inj}(p_k))$ from $L_k$, for $k=1,2.$
Since $s({\rm inj}(p_1))+s({\rm inj}(p_2))$ is strictly less than
$\rho (L_1,L_2)$, there is an interval worth of points $q$ for which
\begin{equation}\label{eq: injectivity and projection}
\rho (L_1,q) >s({\rm inj}(p_1)) \text{ and }  \rho (L_2,q) >s({\rm inj}(p_2)).
\end{equation}
Pick one such point $q$ and consider the bisector $B_{q}$ (See Figure \ref{fig:bisector}).
By Lemma \ref{lem: bisector properties}, the boundary of a  bisector  divides $\partial \mathbb{H}_{\mathbb{C}}^2$ into two open balls. Let $U_1$ be the open ball in the complement of $\partial B_{q_1}$ that contains $\partial L_1$.
Similarly, define $U_2$ to be the open ball in the complement of $\partial B_{q_2}$ that contains $\partial L_2$. Let $V_k$ be the open ball in the complement of ${\partial}B_q$ that contains $\partial L_k$, for $k=1,2$.
Now, by construction, orthogonal projection of $B_{q_1}$ to $L_1$ is the  disc centered at $p_1$ of radius $s({\rm inj}(p_1))$. Since non-trivial elements of $\Gamma_1$ must move  the interior of this disc away from itself, we can conclude  by Lemma \ref{lem: bisector properties},  that any non-trivial $\Gamma_1$-translate of $B_{q_1}$ is disjoint from $B_{q_1}$.  This in turn implies that
 $\gamma (V_2) \subset U_1$, for any non-trivial $\gamma \in \Gamma_1$ (see Figure \ref{fig:combination}). Similarly,   $\gamma (V_1) \subset U_2$, for any non-trivial $\gamma \in \Gamma_2$.
Thus the interior of  $U_1^{c} \cap U_2^{c}$  is precisely invariant under the action of the identity in $\Gamma$. Hence, by the free product version of the  combination theorem (the ping-pong lemma), $\Gamma$ is discrete and the abstract free product of $\Gamma_1$ and $\Gamma_2$.

\begin{figure}[htb]
 \begin{center}
     \includegraphics[width=12cm]{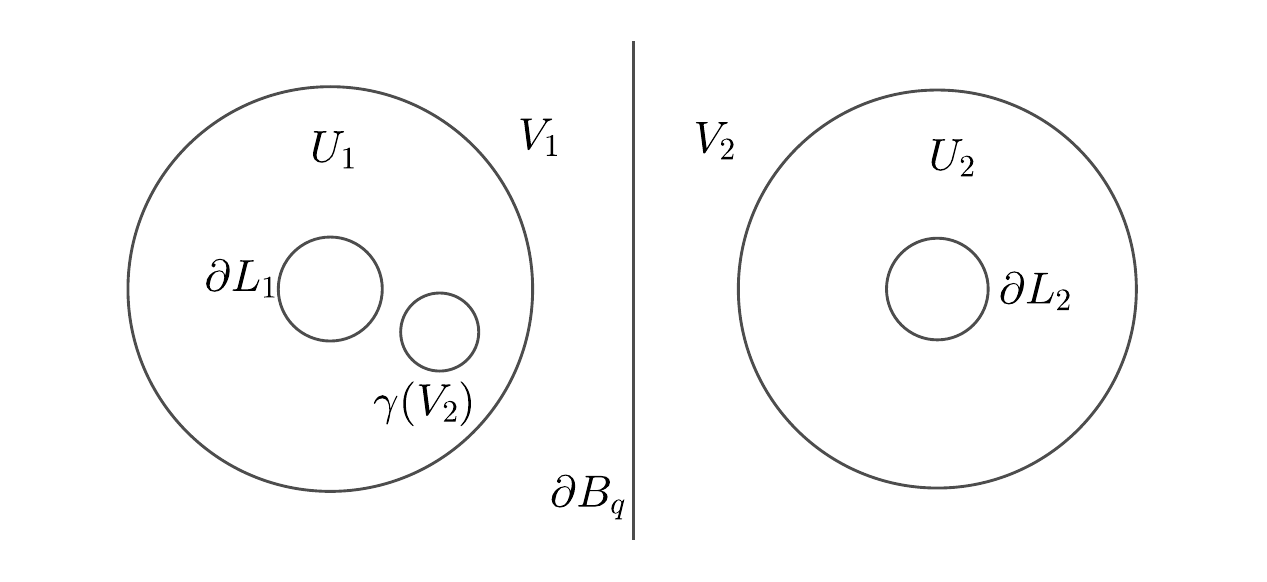}
    \caption{$\gamma (V_2) \subset U_1$}
    \label{fig:combination}
  \end{center}
\end{figure}

Item (2) is equivalent to $\partial L_k$ being precisely invariant under $\Gamma_k$ in $\Gamma$.
Noting that $\partial L_1$ is kept invariant by $\Gamma_1$, and a non-trivial element  of $\Gamma_2$ moves $\partial L_1$ into $U_2$, an induction argument on word length in $\Gamma$ verifies item (2).

To prove  Item (3), note that it is enough to show
\begin{equation}\label{eq: distance realized}
\rho (L_1,L_2) \leq \rho (L_1, \gamma L_2), \text{ for all }\gamma \in \Gamma.
\end{equation}
We first observe that for any non-trivial $\gamma \in \Gamma$,
$\gamma B_{q}$ separates $L_1$ from $\gamma L_2$, and
thus
$$
\rho(L_1,\gamma L_2) \geq \rho(L_1, \gamma B_{q})
+\rho(\gamma B_{q}, \gamma L_2) =\rho(L_1, \gamma B_{q})
+\rho(B_{q}, L_2).
$$
Now, since
$$
\rho(L_1, L_2)=\rho(L_1, B_{q})+ \rho(B_{q}, L_2),
$$
it is enough to show that
\begin{equation}\label{eq: distance to bisector}
\rho(L_{k}, \gamma B_{q})\geq \rho(L_k, B_{q}), \text{ for } k=1,2.
\end{equation}

We  prove (\ref{eq: distance to bisector}) for $k=1$.  Note that an element of $\Gamma_1$ takes
$B_{q}$ into  the half-space bounded by $B_{q_1}$ and keeps the distance to $L_1$ the same.  A non-trivial element of $\Gamma_2$, moves $B_q$ into the half-space bounded by $B_{q_2}$ which
contains $L_2$, hence increasing the distance to $L_1$. Lastly,
 induction on the word length of the element $\gamma$ in
(\ref{eq: distance to bisector}) finishes the argument for $k=1$.
The argument for $k=2$  works the same way.
This finishes the proof of Item (3).
\end{proof}

\subsection{Tube function bounds}
In this subsection we put to use the geometric combination theorem.
Let
\begin{equation*}
d_{\chi}=\sup \{\epsilon : \text{if } S \hookrightarrow X
\text{~and~ $\chi(S)=\chi$, then $S$ has a tube of width } \epsilon   \}
\end{equation*}
and recall the function
\begin{equation*}
s(x) = 2{\rm arcsinh}\left( \frac{1}{{\rm sinh}\frac{x}{2}} \right).
\end{equation*}

\begin{cor} \label{cor: optimal tube function}
The optimal tube  function satisfies
\begin{equation} \label{eq: inequality}
\frac{1}{4}\log\left( \frac{1}{\pi |\chi |}+1 \right)\leq d_{\chi} \leq
s\left({\rm arccosh}\left(\cot \frac{\pi/2}{|\chi | +2}\right)\right).
\end{equation}
In particular,
\begin{equation}\label{eq: asymptotics}
\frac{1}{4{\pi}|\chi |} \lesssim d_{\chi} \lesssim
\frac{2\sqrt{\pi}}{\sqrt{|\chi |}}, \text{ as } |\chi| \rightarrow \infty.
\end{equation}
\end{cor}

Here  the notation $f(x) \lesssim g(x)$ means there exists a constant $C >0$  so that $f(x) \leq C g(x)$, for $x$ large.

\begin{proof} The left-hand inequality follows from Theorem $\text{A}$.
The rest of the proof concerns the right-hand  inequality.
For genus $g \geq 2$, there exists a closed hyperbolic surface $S_g$ of genus $g$  containing an embedded disc of radius
${\rm arccosh}(\cot \frac{\pi}{4g})$ (see e.g. \cite{B-tubular}).
Fix small $\epsilon >0$.
Let $\Gamma_1$ and $\Gamma_2$ be two $\mathbb{C}$-Fuchsian groups whose quotients are isometric to $S_g$, and whose invariant complex lines  are $L_1$ and $L_2$, respectively.
Next we move the configuration so that
\begin{itemize}
\item the orthogeodesic between $L_1$ and $L_2$ has endpoints at
$p_1$ and $p_2$ where,
$$
{\rm inj}(p_1)={\rm inj}(p_2)={\rm arccosh}\left(\cot \frac{\pi}{4g}\right)
$$
\item the distance
$$
\rho(L_1,L_2) =s({\rm inj}(p_1))+s({\rm inj}(p_2)) +\epsilon
$$
\end{itemize}

The groups $\Gamma_1$ and $\Gamma_2$ satisfy the hypotheses of Theorem \ref{thm: combination theorem}.  Hence,
setting $\Gamma=<\Gamma_1,\Gamma_2>$,
 we have constructed a
complex hyperbolic $2$-manifold $X=\mathbb{H}^{2}_{\mathbb C}/\Gamma$ with two embedded complex
geodesic surfaces of Euler characteristic $\chi = \chi(S_g)$ with a  distance
\begin{equation*}
s({\rm inj}(p_1))+s({\rm inj}(p_2)) +\epsilon
\end{equation*}
from each other. Now the optimal tube function must satisfy
\begin{displaymath}
d_{\chi} \leq \frac{1}{2}\left[s({\rm inj}(p_1))+s({\rm inj}(p_2)) +\epsilon
\right].
\end{displaymath}
for all $\epsilon >0$.
Letting $\epsilon$ go to zero, we conclude that
\begin{equation*}
d_{\chi} \leq \frac{1}{2}\left[ s({\rm inj}(p_1))+s({\rm inj}(p_2))\right]=
s\left({\rm arccosh}\left(\cot \frac{\pi}{4g}\right)\right).
\end{equation*}
Setting $g= \frac{|\chi |}{2}+1$ yields the right-hand inequality of
(\ref{eq: inequality}).
Expression  (\ref{eq: asymptotics}) is straightforward and left to the reader.
\end{proof}

\begin{rem}
In (\ref{eq: asymptotics}) we have shown upper and lower bounds for the asymptotic growth rate of  the optimal tube function.
These  inequalities    can be used to get lower and upper bounds on the volume of a  tubular neighborhood of a complex geodesic surface of width of the tube  function. In particular,  the asymptotic volume of the tubular neighborhood goes to zero when the width of the tube  used  is the left-hand expression in (\ref{eq: asymptotics}).  On the other hand, using the right-hand expression yields a volume that is bounded from above and  does not go to zero.  Hence we finish with the  question:
\vskip10pt
{\bf Question:} What is the  asymptotic  growth rate of the optimal  tube function, and what is the asymptotic growth rate of the volume of a tubular neighborhood with this optimal tube width?

\end{rem}



\begin{thebibliography}{1}


\bibitem{B-pants} A. Basmajian,
   \emph{Constructing pairs of pants.}   Ann. Acad. Sci. Fenn. Ser. A I Math.  15  (1990),  no. 1, 65-74.

 \bibitem{B-Bulletin}A. Basmajian,
   \emph{Generalizing the hyperbolic collar lemma}.
   Bull. Amer. Math. Soc. (N.S.), 27 (1992),
   no. 1, 154--158.


\bibitem{B-stable} A. Basmajian,
  \emph{The stable neighborhood theorem and lengths of closed geodesics.}
     Proc. Amer. Math. Soc.  119  (1993),  no. 1, 217-224.

\bibitem{B-tubular} A. Basmajian,
   \emph{Tubular neighborhoods of totally geodesic hypersurfaces in hyperbolic manifolds.}
    Invent. Math.  117  (1994),  no. 2, 207-–225.

\bibitem{Bas-Miner} A. Basmajian and  R. Miner,
  \emph{Discrete subgroups of complex hyperbolic motions.}
    Invent. Math. 131 (1998), no. 1, 85--136.

\bibitem{B-Br-Ha-Me}  A. Basmajian, J. Brisson, A. Hassannezhad,
and A. Metr\'as, \emph{Tubes and Steklov eigenvalues in negatively
curved manifolds.} arXiv.

\bibitem{Bas-Wolpert}
   A. Basmajian and S. A. Wolpert,
   \emph{Hyperbolic 3-Manifolds with nonintersecting closed geodesics.}
    Geometriae Dedicata 97, 251-257 (2003). 


\bibitem{Bey-Poz} J. Beyrer and B. Pozzetti,
\emph{A collar lemma for partially hyperconvex surface group
              representations.}
 Transactions of the American Mathematical Society, 374,
2021, 10, 6927--6961.



\bibitem{Borel-Harish} A. Borel and Harish-Chandra,
 \emph{Arithmetic subgroups of algebraic groups.}
  Ann. of Math. (2)  75,  1962, 485-–535.


\bibitem{Brooks-M} R. Brooks and  J. P. Matelski,
\emph{Collars in Kleinian groups.}
Duke Math. J. 49(1), 163--182.

\bibitem{Buser} P. Buser,
   \emph{The collar theorem and examples.}
   Manuscripta Math.  25  (1978), no. 4, 349-–357.

\bibitem{Buser-book} P Buser,
\emph{Geometry and spectra of compact Riemann surfaces.}
Progress in Mathematics 106, Birkhäuser, Boston (1992).



\bibitem{Cao-Parker} W. Cao  and J.R. Parker,
   \emph{J\o rgensen's inequality and collars in {$n$}-dimensional
              quaternionic hyperbolic space.}
 The Quarterly Journal of Mathematics,
62, 2011, 3, 523--543.



\bibitem{Cartwright-Steger} D. I. Cartwright and T. Steger,
  \emph{Enumeration of the 50 fake projective planes.}
    C. R. Math. Acad. Sci. Paris  348  (2010),  no. 1-2, 11–-13.

\bibitem{Cha} I. Chavel, \emph{Riemannian geometry.}
  Cambridge Studies in Advanced Mathematics, 98, second edition,
     Cambridge University Press, Cambridge, 2006, xvi+471.

\bibitem{Cha2} I. Chavel, \emph{Eigenvalues in Riemannian geometry.}
Second edition, New York,  Academic  Press, 1984.


\bibitem{Chavel-Feldman}
I. Chavel and E. A. Feldman,
\emph{Cylinders on surfaces.}
 Commentarii Mathematici Helvetici 53, 439-447 (1978).




\bibitem{Chinburg-Stover}
 T. Chinburg and M.Stover,
 \emph{Geodesic curves on Shimura surfaces.} Topology Proceedings, Vol 52 (2018), 113--121.

 \bibitem{Dryden-Parlier} E. B. Dryden and H. Parlier,
    \emph {Collars and partitions of hyperbolic cone-surfaces.}
   Geometriae Dedicata,127, 2007, 139--149.
   
 \bibitem{Gallo} D. Gallo, 
   \emph{A {$3$}-dimensional hyperbolic collar lemma},
Kleinian groups and related topics (Oaxtepec, 1981),
Lecture Notes in Math., 971, 31--35, Springer, Berlin-New York,
    1983.

\bibitem{Gilman} J. Gilman,
    \emph{A geometric approach to {J}\o rgensen's inequality}, 
  Advances in Mathematics, 85, 1991, 2, 193--197.

\bibitem{Goldman} W. Goldman,
  \emph{Complex hyperbolic geometry.}
  Oxford Mathematical Monographs. Oxford University Press, New York, 1999. xx+316 pp.


\bibitem{Go-Ka-Le}  W. M.Goldman, M. Kapovich and  B. Leeb,  \emph{Complex hyperbolic manifolds homotopy equivalent to a Riemann surface.} Communications in Analysis and Geometry, 9 (2001), 61--95.

\bibitem{Keen} L. Keen,
   \emph{Collars on Riemann surfaces.}
 Discontinuous groups and Riemann surfaces (Proc. Conf., Univ. Maryland, College Park, Md., 1973),
 pp. 263–-268. Ann. of Math. Studies, No. 79, Princeton Univ. Press, Princeton, N.J., 1974.

 \bibitem{Kim}  D. Kim,
     \emph{Collar lemma in quaternionic hyperbolic manifold.}
  Bull. Korean Math. Soc., 43, 2006, 2, 411--418.

    \bibitem{Kojima} S. Kojima,
   \emph{Immersed geodesic surfaces in hyperbolic {$3$}-manifolds.}
  Complex Variables. Theory and Application. An International
              Journal, 29, 1996, 1, 45-58.

\bibitem{Lee-Zhang} G. Lee and T. Zhang,
 \emph{Collar lemma for Hitchin representations.}
  Geometry $\&$ Topology 21 (2017) 2243-–2280.

\bibitem{Lima} E. L. Lima,
 \emph{The Jordan-Brouwer Separation Theorem for Smooth Hypersurfaces}.
  The American Mathematical Monthly, Vol. 95, No. 1 (Jan., 1988), pp. 39--42.


 \bibitem{Markham-Parker} S. Markham and John R. Parker,
   \emph{Collars in complex and quaternionic hyperbolic manifolds.}
     Special volume dedicated to the memory of Hanna Miriam Sandler (1960--1999),  Geometriae Dedicata, 97, 2003, 199--213.

\bibitem{Maskit} B. Maskit,  \emph{Kleinian groups.}
  Springer-Verlag, Berlin, 1988.

\bibitem{Parker} J. R. Parker,  \emph{Notes on Complex Hyperbolic Geometry}.


\bibitem{Parker-volume} J. R. Parker,
\emph{On the volumes of cusped, complex hyperbolic manifolds and orbifolds.}
Duke Math. J.  94  (1998),  no. 3, 433–-464.

\bibitem{Parlier} H. Parlier, \emph{A note on collars of simple closed geodesics}, Geometriae Dedicata, 112, 2005, 165--168.

\bibitem{Randol} B. Randol, \emph{Cylinders in {R}iemann surfaces},
  Commentarii Mathematici Helvetici, 54, 1979, 1, 1--5.




\end{thebibliography}
\end{document}